\documentclass[11pt,reqno]{amsproc}
\usepackage{amsmath}
\usepackage{amsfonts}
\usepackage{amssymb}
\usepackage{amsthm,color}
\usepackage{newlfont}
\usepackage{bbm}

\usepackage{fancyhdr}
\usepackage{comment}

\usepackage[margin=1in]{geometry}
\usepackage{times}
\usepackage{graphicx, mathtools}

\usepackage[usenames,dvipsnames,svgnames,table]{xcolor}
\usepackage[colorlinks=true, pdfstartview=FitV, linkcolor=blue, citecolor=blue, urlcolor=blue]{hyperref}

\definecolor{labelkey}{rgb}{0,0,0}

\newtheorem{Theorem}{Theorem}[section]

\newtheorem{Lemma}[Theorem]{Lemma}
\newtheorem{Corollary}[Theorem]{Corollary}
\newtheorem{Remark}[Theorem]{Remark}

\numberwithin{equation}{section}

\begin{document}

\def\le{\left}
\def\r{\right}
\def\cost{\mbox{const}}
\def\a{\alpha}
\def\d{\delta}
\def\ph{\varphi}
\def\e{\epsilon}
\def\la{\lambda}
\def\si{\sigma}
\def\La{\Lambda}
\def\B{{\cal B}}
\def\A{{\mathcal A}}
\def\L{{\mathcal L}}
\def\O{{\mathcal O}}
\def\bO{\overline{{\mathcal O}}}
\def\F{{\mathcal F}}
\def\K{{\mathcal K}}
\def\H{{\mathcal H}}
\def\D{{\mathcal D}}
\def\C{{\mathcal C}}
\def\M{{\mathcal M}}
\def\N{{\mathcal N}}
\def\G{{\mathcal G}}
\def\T{{\mathcal T}}
\def\R{{\mathbb R}}
\def\I{{\mathcal I}}

\def\bw{\overline{W}}
\def\phin{\|\varphi\|_{0}}
\def\s0t{\sup_{t \in [0,T]}}
\def\lt{\lim_{t\rightarrow 0}}
\def\iot{\int_{0}^{t}}
\def\ioi{\int_0^{+\infty}}
\def\ds{\displaystyle}
\def\pag{\vfill\eject}
\def\fine{\par\vfill\supereject\end}
\def\acapo{\hfill\break}
\def\div{\text{div}}

\def\beq{\begin{equation}}
\def\eeq{\end{equation}}
\def\barr{\begin{array}}
\def\earr{\end{array}}
\def\vs{\vspace{.1mm} \\}
\def\rd{\reals\,^{d}}
\def\rn{\reals\,^{n}}
\def\rr{\reals\,^{r}}
\def\bD{\overline{{\mathcal D}}}
\newcommand{\dimo}{\hfill \break {\bf Proof - }}
\newcommand{\nat}{\mathbb N}
\newcommand{\E}{\mathbb E}
\newcommand{\Pro}{\mathbb P}
\newcommand{\com}{{\scriptstyle \circ}}
\newcommand{\reals}{\mathbb R}
\newcommand{\normmm}[1]{{\left\vert\kern-0.25ex\left\vert\kern-0.25ex\left\vert #1 
   \right\vert\kern-0.25ex\right\vert\kern-0.25ex\right\vert}}

\newcommand{\red}[1]{\textcolor{red}{#1}}

\def\Amu{{A_\mu}}
\def\Qmu{{Q_\mu}}
\def\Smu{{S_\mu}}
\def\H{{\mathcal{H}}}
\def\Ime{{\textnormal{Im }}}
\def\Tr{{\textnormal{Tr}}}
\def\E{{\mathbb{E}}}
\def\P{{\mathbb{P}}}
\def\span{{\textnormal{span}}}
\def\fei#1{\textcolor{red}{#1}}
\def\EE{E}
\def\HH{\mathbb{H}}
\def\OO{\mathcal O}
\def\RR{\mathbb R}
\def\abs#1{\left|#1\right|}
\def\paren#1{\left(#1\right)}

\title{The growth mechanism of boundary layers for the 2D Navier-Stokes equations}
\author[F.~Wang]{Fei Wang}
\address{Department of Mathematics, Shanghai Jiao Tong University, Shanghai, 200240}
\email{fwang256@sjtu.edu.cn}
\author[Y. C. Zhu]{Yichun Zhu}
\address{Department of Mathematics, University of Maryland, College Park, MD 20740}
\email{yzhu00@umd.edu}
\date{}

\date{}

\begin{abstract}
	We give a detailed description of formation of the boundary layers in the inviscid limit problem. To be more specific, we prove that the magnitude of the vorticity near the boundary is growing to the size of $1/\sqrt{\nu}$ and the width of the layer is spreading out to be proportional the $\sqrt{\nu} $ in a finite time period. In fact,  the growth time scaling is almost $\nu$.
	\hfill \today
\end{abstract}

\maketitle
\setcounter{tocdepth}{1} 
\tableofcontents

\section{Introduction}
We consider the Cauchy problem for the two dimensional incompressible Navier-Stokes  equations 
\begin{equation}\label{NS1}
\partial_t u -\nu \Delta u + u \cdot\nabla u + \nabla p=0
\end{equation}
\begin{equation}\label{NS2}
\text{div} u =0
\end{equation}
\begin{equation}\label{NS3}
u|_{t=0}=u_0
\end{equation}
on the half-space domain $\mathbb{H}=\mathbb{T} \times \mathbb{R}_+ = \{ (x,y)\in \mathbb{T}\times \mathbb{R}: y \geq 0 \}$, with the periodic boundary condition in $x$, and the no-slip boundary condition
\begin{equation}\label{NS4}
u|_{y=0}=0.  
\end{equation}
Here we take the torus $\mathbb{T}=[-\pi,\pi]$, but our results apply for any torus $\mathbb{T}=[-a,a]$ for $a>0$. The goal of this paper is to study the growth mechanism of the boundary layers in the inviscid limit problem of the Navier-Stokes system for initial data that are analytic only near the boundary of the domain, and are Sobolev smooth away from the boundary. As an easy implication, we know this growing procedure is still valid for analytic initial data or Sobolev initial data with support away from the boundary, which are considered in~\cite{SammartinoCaflisch98a,SammartinoCaflisch98b}, and~\cite{Maekawa14} respectively.

In mathematical fluid dynamics, one of the fundamental problems is to determine whether the solutions of the Navier-Stokes equations \eqref{NS1}--\eqref{NS2} could be described 
to the leading order by  the solutions of the Euler equations
\begin{equation*}
	\begin{aligned}
	&\partial_t u^{\EE}+u^{\EE}\cdot\nabla u^{\EE}+\nabla p^{\EE}=0 
	\\
	&\div u^{\EE}=0
\end{aligned}
\end{equation*}
in the inviscid limit $\nu\to 0$. In ~\cite{Kato84b}, Kato gave a criteria for the inviscid limit property to hold in the topology $L^\infty(0,T,L^2(\mathbb{H}))$, namely, he proved that  the limit of energy dissipation in a thin layer of size $\nu$ 
\begin{equation*}
	\lim_{\nu\to 0} ~\nu \int_0^T \int_{\{y\lesssim \nu\}}|\nabla u|^2 dxdydt=0
\end{equation*}
is equivalent to the strong convergence in the energy norm.
Further refinements and extensions based on Kato's original argument could be found in~\cite{BardosTiti13,
	ConstantinElgindiIgnatovaVicol17,ConstantinKukavicaVicol15,
	Kelliher08,Kelliher17,Masmoudi98,TemamWang97b,Wang01} ; cf.~also the recent review~\cite{MaekawaMazzucato16}. 

However, verifying the proposed properties on the sequence of Navier-Stokes solutions in the above mentioned work based on the knowledge of the initial datum is in general an outstanding open problem. In fact, even the question of whether the weak $L^2_t L^2_x$ inviscid limit holds (against test functions compactly supported in the interior of the domain), remains open. We refer the readers to~\cite{ConstantinVicol18,DrivasNguyen18, ConstantinBrazil18} for conditional results in this direction. 
The main 
difficulty of the inviscid limit problems comes from the contribution from the behavior of the fluid near a  nontrivial boundary. 
To be more specific, in the region close to the boundary $\partial \HH$, the mismatch of boundary conditions between the viscous Navier-Stokes flow (no-slip) and the inviscid Euler flow (slip) is the main obstacle of obtaining a uniform in $\nu$ estimate of the solutions.
Another difficulty is due to the fast growth of vorticity close to the boundary $\partial \HH$ as $\nu \to 0$.

Due to the mismatch of boundary conditions, in general the Navier-Stokes solutions could not be approximated well by Euler solutions near the boundary, In order to study this phenomena, Prandtl proposed the following decomposition
\begin{equation}
	\label{EQ07}
	u(x,y,z, t) = u^{E}(x,y,z,t) + u^{P}\left(x,y,\frac{z}{\sqrt{\nu}},t\right) + \OO(\sqrt{\nu})
\end{equation}
 for the Navier-Stokes solution $u$,
where $u^P$ is the solution of the Prandtl boundary layer equations. So far, the Prandtl equation are very well-understood by an extensive literature . 
For instance, see~\cite{AlexandreWangXuYang14,DietertGerardVaret19,GerardVaretMasmoudi13,IgnatovaVicol16,KukavicaMasmoudiVicolWong14,KukavicaVicol13,LiuYang16,LombardoCannoneSammartino03,MasmoudiWong15,Oleinik66,SammartinoCaflisch98a} for the well-posedness and~\cite{GerardVaretDormy10,GuoNguyen11,GerardVaretNguyen12,LiuYang17} for the ill-posedness. 

However, establishing the validity of the expansion \eqref{EQ07} only making regularity assumption on the initial data without any symmetry  presents a number of outstanding challenges and fewer results are available. In their seminal works~\cite{SammartinoCaflisch98a,SammartinoCaflisch98b} in 1998,
Sammartino-Caflisch establish the validity of the Prandtl expansion for analytic initial data in the half space $\HH$ in dimension two. While in 2014, Maekawa (\cite{Maekawa14})  proved that the inviscid
limit also holds for initial datum with Sobolev regularity whose associated vorticity is supported  away from the boundary, taking advantage of the weak interaction between the boundary vorticity and the bulk flow inside the domain. For different proofs of the Caflisch-Sammartino and 
Maekawa results in 2D and 3D, we refer to~\cite{WangWangZhang17, FeiTaoZhang16,FeiTaoZhang18, NguyenNguyen18}. 

The gap between the Sammartino-Caflisch~\cite{SammartinoCaflisch98a,SammartinoCaflisch98b} and the Maekawa~\cite{Maekawa14} results was closed by Kukavica, Vicol, and the first author~\cite{KukVicWan20, KukVicWan20a}, in which we established the strong inviscid limit in the energy norm, for initial data that is only analytic close to the boundary of the domain, and has finite Sobolev regularity in the complement. Later the first author \cite{Wan20} addressed the problem in the 3D case with initial data only analytic near the boundary. 
 Up to now, the class of initial data considered in \cite{KukVicWan20, KukVicWan20a, Wan20} appears to be the largest class of initial data that the strong inviscid limit is known to hold, in the absence of structural or symmetry assumptions. We also mention that these results was extended to a general bounded domain in $\RR^2$ by Bardos, Nguyen, Nguyen and Titi in \cite{BarNguNguTit21}.  Furthermore, the validity of the Prandtl expansion under such kind of class was established in~\cite{KukNguVicWan21} by Kukavica, Nguyen, Vicol, and Wang, which close the inviscid limit problem completely for this functional setting only analytic near the boundary.
 
 Recently in~\cite{GerardVaretMaekawaMasmoudi16, GerMaeMas20}, Gerard-Varet, Maekawa, and Masmoudi proved the stability of Gevrey perturbations in $x$ and Sobolev perturbation in $y$ for shear flows of the Prandtl type, when the Prandtl shear flow is both monotonic and concave.  Lastly, we mention that the vanishing viscosity limit is also known to hold in the presence of certain symmetry assumptions on the initial data, which is maintained by the flow; see e.g.~\cite{BonaWu02,HanMazzucatoNiuWang12,Kelliher09,LopesMazzucatoLopes08,LopesMazzucatoLopesTaylor08,MaekawaMazzucato16,Matsui94,MazzucatoTaylor08,GKFMNL17} and references therein. Also, the very recent works~\cite{GerardVaretMaekawa18,GuoIyer18,GuoIyer18a,Iyer18} establish the vanishing viscosity limit and the validity of the Prandtl expansion for the stationary  Navier-Stokes equation, in certain regimes.

In the previous works, a lot of attention is paid towards the direction of either verifying the convergence of Navier-Stokes system to Euler equations in an energy norm or proving the Prandtl expansion for the Navier-Stokes equations. However, the formation of the boundary layers is not well studied to the knowledge of the authors, which is one of key elements of the inviscid limit problem. From the perspectives of physics, it seems unlikely that the boundary is immediately there if it is absent initially. Intuitively, there should be a process where the magnitude of the vorticity near the boundary is growing to the size of $1/\sqrt{\nu}$ and the width of the layer is spreading out to be proportional the $\sqrt{\nu} $ eventually. In this paper, we verify this conjecture and prove that the growth time scaling is almost $\nu$. One interesting phenomena in our proof is that the size of the vorticity and the width of the layer seem to grow linearly in time. However, according to~\cite{KukVicWan17}, we believe this should not be the optimal result and conjecture that the growing speed should be $1/(T-t)$ if the initial vorticity is huge. 

As mentioned in the beginning, we aim to study the growth mechanism of the boundary layers for initial data that are analytic only near the boundary of the domain, and are Sobolev smooth away from the boundary. Typically our result is more interesting when there is no boundary layer initially. On the other hand, whether the growing boundary layer is monotonic and concave is still open so far, which is directly relevant to the results in~\cite{GerardVaretMaekawaMasmoudi16, GerMaeMas20} where they assume the initial boundary layer is so. Compared with the previous work~\cite{KukVicWan20, KukVicWan20a, Wan20}, the main novelty here is to design a time dependent weight function (see~\eqref{weight}) to capture the involving feature of the boundary layer. However, this new weight is not compatible with the norm we used in the precious work. Modifying the time coefficient in the norm which recovers the norm by integral accordingly is one of the main difficulty in this paper. Also direct estimate in the proof will not give the optimal scaling of time as $\nu$, instead we adjust the time coefficient in the analytic norm to get the almost optimal scaling, see Remark \ref{r}.
\section{The functional setting and main results}

\subsection{Notations}
Throughout this paper, we will use $f_{\xi} \in \mathbb{C}$ to denote the Fourier transform of $f(x,y)$ with respect to the $x$ variable  at frequency $\xi \in \mathbb{Z}$, i.e., $f(x,y)= \sum_{\xi \in \mathbb{Z}} f_{\xi}(y) e^{i x \xi }$. We also write $u_{i,\xi}$ or $(u_i)_{\xi}$ for the Fourier transformation of $u_i$ where $i=1,2$. Let $\mathfrak{Re}z$ and $\mathfrak{Im}z$ be the real part and imaginary part of a complex number $z$ respectively. We define the complex domain $\Omega_{\mu}$ by 
\[
\Omega_{\mu}= \{ z \in \mathbb{C} : 0 \leq \mathfrak{Re}z \leq 1,\ \vert \mathfrak{Im} z \vert \leq \mu \mathfrak{Re}z  \}  \cup \{ z \in \mathbb{C}: 1 \leq \mathfrak{Re}z \leq 1+\mu,\ \ \vert \mathfrak{Im} z \vert \leq 1+ \mu - \mathfrak{Re}z  \}, 
\]
where we assume that $\mu<\mu_0$ for some small constant $\mu_0\in (0, 1/10)$.
For simplicity, we write  $e^{\e_0(1+\mu- z)_{+}|\xi|}$ to denote $e^{\e_0(1+\mu-\mathfrak{Re} z)_{+}|\xi|}$ for a complex number z without causing any confusion. We use $f\lesssim g$ to represent $f \leq Cg$ for some constant $C$.

\subsection{Functional settings}
Define the weight function
\begin{equation}
	\label{weight}
w_t(y)=
\begin{cases}
w\left(y \left(\frac{\nu^2}{t \gamma^2}\right)^{\beta}\right) + \left(1-\frac{\gamma^2 t}{ \nu^2}\right)_{+},\  \ \ \ 0<t \leq \nu^2/\gamma^2\\
w(y),\ \ \ \ \ \ \ \ \ \ \ \ \ \ \ \	\	\	\	\	\  \ \ \ \ \ \ \ \ 	\	\ \ \ \ \ \ \  t>\nu^2/\gamma^2
\end{cases}
,
\end{equation}
where
\begin{equation}
\begin{aligned}
w(y)=
\begin{cases}
\sqrt{\nu},\ &0<y \leq \sqrt{\nu}\\
y,\ &\sqrt{\nu} \leq y \leq 1\\
1,\ &1 \leq y \leq 1+ \mu_0
\end{cases}
\end{aligned}
,
\end{equation}
 $\beta \in (0, 1/4)$, and $\gamma > \gamma_0$ is some sufficiently large constant to be determined. When $t=0$, we set $w_t(y)=1$.
As in~\cite{KukVicWan20}, we introduce a weighted analytic norm
\begin{equation*}
\Vert f \Vert_{\mathcal{L}^{\infty}_{\mu,\nu,t}} = \sup_{y \in \Omega_\mu} w_t(y) |f(y)|
\end{equation*}
for a complex function $f$, where we use $ w_t( y)$ to stand for $ w_t(\mathbb{R}e\  y)$ for the sake of simplicity. For $\e_0 \in (0,1)$  sufficiently small to be determined below, we set
\begin{equation*}
\Vert f \Vert_{X_{\mu,t}} = \sum_{\xi \in \mathbb{Z}}  \Vert e^{\e_0(1+ \mu -y)_+ |\xi|} f_{\xi} \Vert_{\mathcal{L}^{\infty}_{\mu,\nu,t}},
\end{equation*}
and
\begin{equation}\label{X(t)}
\Vert f \Vert_{X(t)}= \sup_{\mu < \mu_0- \gamma \sqrt{t}} \left(\sum_{0 \leq i+j \leq 1} \Vert \partial_x^i (y \partial_y)^j  f\Vert_{X_{\mu,t}} + \sum_{i+j=2} (\mu_0 - \mu - \gamma \sqrt{t})^{1/2+ \a} \Vert \partial_x^i (y \partial_y)^j f \Vert_{X_{\mu,t}}\right),
\end{equation}
where $\a \in (0, 1/2)$. Similarly as in~\cite{KukVicWan20}, we define ${L}^1$ based norms
\begin{equation*}
\Vert f \Vert_{\mathcal{L}^1} = \sup_{0 \leq \theta < \mu} \Vert f \Vert_{L^1(\partial \Omega_{\theta})},
\end{equation*}
\begin{equation*}
\Vert f \Vert_{Y_{\mu}} = \sum_{\xi} \Vert e^{\e_0(1+\mu-y)_+|\xi|} f_{\xi} \Vert_{\mathcal{L}^1},
\end{equation*}
and
\begin{equation*}
\Vert f \Vert_{Y(t)} = \sup_{\mu < \mu_0- \gamma t} \left(\sum_{0 \leq i+j \leq 1} \Vert  \partial_x^i (y \partial_y)^j f \Vert_{Y_{\mu}}+ \sum_{i+j =2} (\mu_0 - \mu - \gamma t)^{\a} \Vert   \partial_x^i (y \partial_y)^j f \Vert_{Y_\mu}\right).
\end{equation*}
We also need a weighted $L^2$ norm
\begin{equation*}
\Vert f\Vert_S^2= \Vert yf \Vert^2_{L^2(y \geq 1/2)}= \sum_{\xi} \Vert yf_{\xi} \Vert_{L^2(y \geq 1/2)}^2,
\end{equation*}
and a weighted version of the Sobolev $H^3$ norm
\begin{equation*}
\Vert f \Vert_{Z} =\sum_{0 \leq i+j \leq 3} \Vert \partial_x^i \partial_y^j f \Vert_S.
\end{equation*}
Furthermore, it is convenient to introduce
\begin{equation*}
\Vert f\Vert_{S_\mu}=\sum_{\xi} \Vert y f_{\xi} \Vert_{L^2(y \geq 1+\mu)}.
\end{equation*}
At last, we define the cumulative time-dependent norm by
\begin{equation*}
\normmm{\omega}_t= \Vert \omega \Vert_{X(t)} + \Vert \omega \Vert_{Y(t)} + \Vert \omega \Vert_{Z}.
\end{equation*}

\subsection{Main results}
In this section, we state our main theorems and the proof will be presented in Section \ref{Proof}.
\begin{Theorem}\label{thm0}
Let $\mu_0>0$ and assume that $\omega_0$ is such that
\begin{equation}
\sum_{i+j \leq 2} \Vert \partial_x^i (y\partial_y)^j \omega_0\Vert_{X_{\mu_0,0}} + \sum_{i+j \leq 2} \Vert (\partial_x)^i (y\partial_y)^j \Vert_{Y_{\mu_0}} + \sum_{i+j\leq 3} \Vert \partial_x^i (y \partial_y)^j \omega_0 \Vert_S \leq M< \infty.
\end{equation}
Then there exists a sufficiently large $\gamma>0$ and a time $T>0$, depending on $M$ and $\mu_0$, such that the solution $\omega$ to the system~\eqref{NS1}-\eqref{NS3} satisfies 
\begin{equation}
\sup_{t \in [0,T]} \normmm{\omega(t)}_t \leq CM.
\end{equation}
\end{Theorem}
\begin{Remark}\label{r}
In the above result, the time scaling for the growth procedure is $t\sim \nu^2$, which is far from optimal. 
By redefining the weight function
\begin{equation}
w_t(y)=
\begin{cases}
1 ,& t=0\\
w\left(y \left(\frac{1}{t}\left(\frac{\nu}{ \gamma}\right)^{2^n/(2^n-1)}\right)^{\beta}\right) + \left(1-\left(\frac{\gamma}{ \nu}\right)^{2^n/(2^n-1)}t\right)_{+},\ \ \   & 0< t \leq (\nu/\gamma)^{{2^n}/{2^n-1}}\\
w(y),&  t>(\nu/\gamma)^{2^n/2^n-1}
\end{cases}
\end{equation}
and the analytic norm
\begin{equation}
	\Vert f \Vert_{X(t)}= \sup_{\mu < \mu_0- \gamma \sqrt{t}} \left(\sum_{0 \leq i+j \leq 1} \Vert \partial_x^i (y \partial_y)^j  f\Vert_{X_{\mu,t}} + \sum_{i+j=2} \left(\mu_0 - \mu - \gamma t^{1/2^n}\right)^{1/2+ \a} \Vert \partial_x^i (y \partial_y)^j f \Vert_{X_{\mu,t}}\right),
\end{equation}
following the proof in Section \ref{Proof} and the appendix, we can upgrade the time scaling to $t\sim\nu^{1+\e}$ for arbitrarily small $\e$, which is almost optimal.
 
\end{Remark}
The above result immediately implies the following inviscid limit for 2-D Navier-Stokes equation.
\begin{Theorem}\label{thm1}
Let $\omega_0$ and $T>0$ be as in Theorem \ref{thm0}. Denote by $u^{\nu}$ the solution of the Navier-Stokes equations \eqref{NS1}-\eqref{NS4} with viscosity $\nu>0$. Also denote by $\bar{u}$ the solution of the Euler equations with initial datum $u_0$. Then we have
\begin{equation*}
\lim_{\nu \to 0} \sup_{t \in [0,T]} \Vert u^{\nu}- \bar{u}(t)\Vert_{L^2(\mathbb{H})} =0.
\end{equation*}
\end{Theorem}

\section{Preliminaries}
The following properties for the weight function $w(y)$ will be frequently used throughout the proof. For more details see Remark 2.1 in \cite{KukVicWan20}.
\begin{Lemma}[I.~Kukavica, V.~Vicol, and F.~Wang]\label{l0}
The weight function $w(y)$ satisfies the following properties:
\begin{enumerate}
\item
$w(y) \lesssim w(z)$ for $y \leq z$,
\item
$w(y) \lesssim w(z)$ for $0 < y/2 \leq z \leq 1+ \mu_0$,
\item
$\sqrt{\nu} \lesssim w(y) \lesssim 1$ for $y \in [0,1+\mu_0]$,
\item
$y \lesssim w(y)$ for $y \in [0,1+\mu_0]$,
\item
$w(y)e^{-\frac{y}{C \sqrt{\nu}}} \lesssim \sqrt{\nu}$ for $y \in [0, 1+ \mu_0]$ where $C>0 $ is a sufficiently large constant.
\end{enumerate}
\end{Lemma}
We also recall the following properties for the Green function of 2D Navier-Stokes equation which can be found in Section 3 in \cite{NguyenNguyen18}. 
After taking the Fourier transform in the tangential $x$ variable, the Navier-Stokes equations may be rewritten as 
\begin{equation}\label{NS}
	\begin{aligned}
		\partial_t \omega_{\xi} - \nu \Delta_{\xi} \omega_{\xi} = N_{\xi}\\
		\nu (\partial_y + |\xi|) \omega_{\xi}= B_{\xi},
	\end{aligned}
\end{equation}
for $\xi \in \mathbb{Z}$, where
\begin{equation*}
	N_{\xi}(s,y) = -(u \cdot \nabla \omega)_{\xi}(s,y),
\end{equation*}
\begin{equation*}
	B_{\xi}(s) = (\partial^{-1}(u \cdot \nabla \omega))_{\xi}(s)|_{y=0} = -(\partial_y \Delta^{-1}_{\xi}N_{\xi}(s))|_{y=0},
\end{equation*}
and
\begin{equation*}
	\Delta_{\xi}=-\xi^2 + \partial_y^2.
\end{equation*}
Denoting the Green's function for this system by $G_{\xi}(t,yz)$, we may represent the solution of this system as 
\begin{equation*}
	\begin{aligned}
		\omega_{\xi}(t,y) =& \int_0^{\infty} G_{\xi}(t,y,z) \omega_{0\xi}(z) dz + \int_0^t \int_0^{\infty} G_{\xi}(t-s,y,z) N_{\xi}(s,z) dz ds\\
		&+ \int_0^t G_{\xi}(t-s,y,0) B_{\xi}(s) ds,
	\end{aligned}
\end{equation*}
where $\omega_{0\xi}(z)$ is the Fourier transform of the initial data.
The next Lemma gives an estimate of the Green's function $G_{\xi}$ of the Stokes system.
\begin{Lemma}[\cite{NguyenNguyen18}]
The Green's function $G_{\xi}$ for the system \ref{NS} is given by
\begin{equation*}
G_{\xi}=\tilde{H}_{\xi} + R_{\xi},
\end{equation*}
where 
\begin{equation}\label{kernel representation}
\tilde{H}_{\xi}(t,y,z) = \frac{1}{\sqrt{\nu t}} (e^{-\frac{(y-z)^2}{4 \nu t}} + e^{-\frac{(y+z)^2}{4 \nu t}}) e^{- \nu \xi^2 t}
\end{equation}
is the one dimensional heat kernel for the half space with homogeneous Neumann boundary condition. The residual kernel $R_{\xi}$ has the property $(\partial_y - \partial_z)R_{\xi}(t,y,z)=0$, meaning that it is a function of $y+z$, and it satisfies the bounds
\begin{equation}\label{kernel estimation}
|\partial_z^k R_{\xi}(t,y,z)| \lesssim b^{k+1} e^{- \theta_0 b (y+z)} + \frac{1}{(\nu t)^{(k+1)/2}} e^{-\theta_0 \frac{(y+z)^2}{\nu t}} e^{-\frac{\nu \xi^2 t}{8}}, \  k \in \mathbbm{N}_0,
\end{equation}
where $\theta_0 >0$ is a constant independent of $\nu$. The boundary reminder coefficient $b$ in \eqref{kernel estimation} is given by
\begin{equation*}
b=b(\xi, \nu) =|\xi| + \frac{1}{\sqrt{\nu}}.
\end{equation*}
The implicit constant in \eqref{kernel estimation} depends only on $k$ and $\theta_0$.
\end{Lemma}
\begin{Remark}
Based on \eqref{kernel estimation}, the residual kernel $R_{\xi}$ satisfies
\begin{equation}\label{kernel estimation 2}
|(y \partial_y)^k R_{\xi}(t,y,z)| \lesssim b e^{-\frac{\theta_0}{2}b(y+z)} + \frac{1}{\sqrt{\nu t}} e^{-\frac{\theta_0}{2}\frac{(y+z)^2}{\nu t}} e^{-\frac{\nu \xi^2 t}{8}},
\end{equation}
for $k \in \{ 0,1,2 \}$, point-wise in $y,z \geq 0$.
\end{Remark}

\section{Estimate of the X-analytic norm for $t \lesssim \nu^2/\gamma^2$}
In this section, we assume $t\leq \nu^2/\gamma^2$, and thus
\begin{equation*}
w_t(y)=w\left(y \left(\frac{\nu^2}{t \gamma^2}\right)^{\beta}\right) + \left(1-\frac{\gamma^2 t}{ \nu^2}\right)_{+}.
\end{equation*}
\begin{Lemma}\label{l1}
Assume that $\mu$ and $\tilde{\mu}$ satisfies $0< \mu < \tilde{\mu} < \mu_0 - \gamma \sqrt{s}$, $\tilde{\mu}- \mu \geq \frac{1}{c} (\mu_0-\mu - \gamma \sqrt{s} )$ for some constant $c \geq 1$. Then we have
\begin{equation}
\begin{aligned}
&\left\Vert e^{\e_0(1+ \mu -y)_{+}|\xi|} (y \partial y) \int_0^{\infty} H_{\xi}(t-s,y,z) N_{\xi}(s,z) dz \right\Vert_{\mathcal{L}_{\mu, \nu, t}^{\infty}}\\
 \lesssim& \left\Vert e^{\e_0(1+\mu-y)_{+}|\xi|} N_{\xi}(s,z) \right\Vert_{\mathcal{L}_{\tilde{\mu},\nu,s}^{\infty}} +  \left\Vert e^{\e_0(1+\mu-y)_{+}|\xi|} (z \partial z)N_{\xi}(s,z) \right\Vert_{\mathcal{L}_{\tilde{\mu},\nu,s}^{\infty}}\\& + \frac{1}{\sqrt{\mu_0-\mu -\gamma \sqrt{s}}} \left\Vert  \partial_z N_{\xi}(s,z) \right\Vert_{L^2(z \geq 1+ \tilde{\mu})}.
\end{aligned}
\end{equation}
\end{Lemma}
\begin{proof}
Let $\phi: \mathbb{R}^+ \to \mathbb{R}^+$ be a smooth, non-increasing fucntion such that $\phi(x)=1$ for $x \in [0, 1/2]$, and $\phi=0$ when $x \geq 3/4$. For $y \in \Omega _\mu$,  by integration by parts, we have
\begin{equation*}
\begin{aligned}
&(y \partial y) \int_0^{\infty} \tilde{H}_{\xi}(t-s,y,z) N_{\xi}(s,z) dz \\
=&-y \int_0^{\infty} \partial_z \tilde{H}_{\xi}(t-s,y,z) N_{\xi}(s,z) dz \\
=&-y \int_0^{\frac{3}{4}y} \phi\left(\frac{z}{y}\right)\partial_z \tilde{H}_{\xi}(t-s,y,z) N_{\xi}(s,z) dz - \int_0^{\infty} \phi'\left(\frac{z}{y}\right) H_{\xi}(t-s,y,z) N_{\xi}(s,z) dz\\
&+ y \int_{\frac{1}{2}y}^{1+\mu} \left(1- \phi\left(\frac{z}{y}\right)\right) H_{\xi}(t-s,y,z) \partial_z N_{\xi}(s,z) dz 
+ y \int_{1+\mu}^{\infty} H_{\xi} (t-s,y,z) \partial_z N_{\xi}(s,z) dz\\
=& I_1 + I_2 + I_3 + I_4.
\end{aligned}
\end{equation*}
 \textbf{(1) Estimate of $\sup_{y \in \Omega_{\mu}} |e^{\e_0(1+\mu-y)_{+}|\xi|} w_t(y) I_1|$:}\\
In this case, $s<t$ and $z < \frac{3}{4}y <y$ hold and we have
\begin{equation*}
\begin{aligned}
 &\left|e^{\e_0(1+\mu-y)_{+}} w_t(y) I_1\right|\\
 \lesssim& \left|\int_0^{\frac{3}{4t}y}|e^{\e_0(1+\mu-z)_{+}} e^{\e_0(z-y)_{+}|\xi|} y \partial_y H_{\xi}(t-s,y,z) \frac{w_t(y)}{w_s(z)} w_s(z) N_{\xi}(s,z) dz\right|\\
\lesssim& \left|\int_0^{\frac{3s}{4t}y} y \partial_y H_{\xi} (t-s,y,z) \frac{w_t(y)}{w_s(z)} dz\right| \cdot \Vert e^{\e_0(1+\mu-y)_{+}} N_{\xi }(s,z)  \Vert_{\mathcal{L}^{\infty}_{\mu, \nu, t}}.
\end{aligned}
\end{equation*}
\textbf{(2) Estimate of $\sup_{y \in \Omega_{\mu}} |e^{\e_0(1+\mu-y)_{+}|\xi|} w_t(y) I_2|$:}\\
For any $y \in \Omega_{\mu}$, we obtain
\begin{equation*}
\begin{aligned}
&\left|e^{\e_0(1+\mu-y)_{+}|\xi|} w_t(y) I_2\right|\\
=& \left|e^{\e_0(1+\mu-y)_{+}|\xi|} w_t(y)\int_0^{\infty} \phi'(\frac{z}{y}) H_{\xi}(t-s,y,z) N_{\xi}(s,z) dz \right|\\
\lesssim& \left\lvert\int_{\frac{1}{2}y}^{\frac{3}{4}y} e^{\e_0(1+\mu-z)_{+} |\xi|} e^{\e_0(z-y)_{+}|\xi|} \frac{w_t(y)}{w_s(z)} H_{\xi}(t-s,y,z) w_s(z) N_{\xi}(s,z) dz\right\rvert\\
\lesssim& \int_0^{\frac{3}{4}y} \frac{w_t(y)}{w_s(z)} H_{\xi}(t-s,y,z) dz  \cdot \left\Vert e^{\e_0(1+\mu -z)_{+}|\xi|} N_{\xi}(s,z)\right\Vert_{\mathcal{L}^{\infty}_{\mu, \nu,s}}.
\end{aligned}
\end{equation*}
To proceed, we derive an estimate of the following term:
\begin{equation}\label{target1}
\left\Vert \mathbbm{1}_{\{0 \leq y \leq 1+\mu\}} \mathbbm{1}_{\{0 \leq z \leq \frac{3}{4}y\}} ( |H_{\xi} (t-s,y,z) + |y \partial_y H_{\xi}(t-s,y,z)|)\frac{w_t(y)}{w_s(z)} \right\Vert_{L^{\infty}_y L^1_z},
\end{equation}
where 
\begin{equation*}
y\partial_y H_{\xi}(t-s,y,z) = \frac{y(y-z)}{2 \nu (t-s)} \frac{1}{\sqrt{\nu (t-s)}} e^{-\frac{(y-z)^2}{4 \nu (t-s)}} e^{-\nu \xi^2 (t-s)}.
\end{equation*}
Since $0 \leq z \leq \frac{3}{4} y$ in equation \eqref{target1}, we have $y \leq 4(y-z)$ and
\begin{equation*}
\begin{aligned}
y \partial_y H_{\xi}(t-s, y,z) \lesssim& \frac{(y-z)^2}{\nu (t-s)} \frac{1}{\sqrt{\nu(t-s)}} e^{-\frac{(y-z)^2}{4 \nu (t-s)}} e^{-\nu \xi^2 (t-s)}\\
\lesssim& \frac{1}{\sqrt{\nu(t-s)}} e^{-\frac{(y-z)^2}{8 \nu (t-s)}} e^{-\nu \xi^2 (t-s)}.
\end{aligned}
\end{equation*}
Therefore, the quantity in \eqref{target1} is bounded by 
\begin{equation*}
\begin{aligned}
&\sup_{0 \leq y \leq 1+\mu} \int_0^{\frac{3}{4} y} \frac{w_t(y)}{w_s(z)} \frac{1}{\sqrt{\nu(t-s)}} \Big( e^{-\frac{(y-z)^2}{4 \nu (t-s)}} +  e^{-\frac{(y-z)^2}{8 \nu (t-s)}} \Big) dz\\
\lesssim& \sup_{0 \leq y \leq 1+\mu}  \int_0^{\frac{3 }{4 } y} \frac{w_t(y)}{w_s(z)} \frac{1}{\sqrt{\nu(t-s)}}   e^{-\frac{(y-z)^2}{8 \nu (t-s)}}  dz.
\end{aligned}
\end{equation*} 
We are now going to estimate the term $w_t(y)/ w_s(z)$. Recall that 
\begin{equation*}
w(y) \lesssim \sqrt{\nu} e^{\frac{y}{C \sqrt{\nu}}}
\end{equation*}
for some sufficiently large $C$. Setting $4c <C$ and noticing that $z < \frac{3}{4}y$, we arrive at 
\begin{equation*}
\begin{aligned}
w\left(y\frac{\nu^{2\beta}}{t^\beta  \gamma^{2\beta}}\right) e^{-\frac{\left(y - z\right)^2}{c^2\nu \left(t-s\right)}} \leq& w\left(y \frac{\nu^{2\beta}}{t^\beta \gamma^{2\beta}}\right) e^{-\frac{y^2}{16c^2\nu\left(t-s\right)}}
\leq w\left(y \frac{\nu^{2\beta}}{t^{\beta} \gamma^{2\beta}}\right) e^{-\frac{y^2}{16c^2 \nu t}}\\
\leq& w\left(y \frac{\nu^{2\beta}}{t^{\beta} \gamma^{2\beta}}\right) e^{-\frac{y}{4c \sqrt{\nu t} }}
\lesssim \sqrt{\nu} e^{\frac{y \nu^{2\beta}}{C \sqrt{\nu} t^\beta \gamma^{2\beta}}}e^{-\frac{y}{4c \sqrt{\nu t} }}\\
\lesssim& \sqrt{\nu} e^{\frac{y }{C \sqrt{\nu} \gamma^{2\beta} t^\beta}}e^{-\frac{y}{4c \sqrt{\nu t}}}.
\end{aligned}
\end{equation*}
Since $\beta \in (0, 1/2)$ and $\gamma> \gamma_0$, for $t \leq T$, there is 
\begin{equation*}
 e^{\frac{y }{C \gamma^{2\beta} \sqrt{\nu} t^\beta}}e^{-\frac{y}{4c \sqrt{\nu t}}} \lesssim e^{\frac{y }{C \gamma_0^{2\beta} \sqrt{\nu} t^{\beta}}}e^{-\frac{y}{4c \sqrt{\nu t}}} \lesssim1
\end{equation*}
for $\nu$ arbitrary small.
Therefore,
\begin{equation}\label{eq1}
w\left(y\frac{\nu^{2\beta}}{t^\beta \gamma^{2\beta}}\right) e^{-\frac{(y - z)^2}{c^2\nu (t-s)}} \lesssim \sqrt{\nu} \lesssim w\left(z\frac{\nu^{2\beta}}{t^\beta \gamma^{2\beta}}\right)
\end{equation}
holds.
We now consider two cases. 
\begin{enumerate}
\item When $s>\nu^2/2\gamma^2$, we have
\begin{equation*}
\begin{aligned}
\frac{w_t(y)}{w_s(z)} = \frac{w\left(y \frac{\nu^{2\beta}}{t^\beta \gamma^{2\beta}}\right)+\left(1-\frac{\gamma^2 t}{\nu^2}\right)_+}{w\left(z \frac{\nu^{2\beta}}{s^\beta \gamma^{2\beta}}\right)+\left(1-\frac{\gamma^2 s}{\nu^2}\right)_+} \leq& \frac{w\left(y \frac{\nu^{2\beta}}{t^\beta \gamma^{2\beta}}\right)}{w\left(z \frac{\nu^{2\beta}}{s^\beta \gamma^{2\beta}}\right)} + \frac{\left(1-\frac{\gamma^2 t}{\nu^2}\right)_+}{w\left(z \frac{\nu^{2\beta}}{s^\beta \gamma^{2\beta}}\right)+\left(1-\frac{\gamma^2 s}{\nu^2}\right)_+}\\
 \leq&  \frac{w\left(y \frac{\nu^{2\beta}}{t^\beta \gamma^{2\beta}}\right)}{w\left(z \frac{\nu^{2\beta}}{s^\beta \gamma^{2\beta}}\right)} + 1.
\end{aligned}
\end{equation*}
By equation \eqref{eq1}, it is easy to get
\begin{equation*}
\frac{w\left(y \frac{\nu^{2\beta}}{t^\beta \gamma^{2\beta}}\right)}{w\left(z \frac{\nu^{2\beta}}{s^\beta \gamma^{2\beta}}\right)}  \leq e^{\frac{(y - z)^2}{c^2\nu (t-s)}}.
\end{equation*}
Imposing $c>4$,  the term \eqref{target1} is bounded by
\begin{equation*}
\begin{aligned}
&\sup_{0 \leq y \leq 1+\mu} \int_0^{\frac{3}{4}y} \frac{w_t(y)}{w_s(z)} \frac{1}{\sqrt{\nu (t-s)}} e^{-\frac{(y-z)^2}{8\nu(t-s)}} dz\\
\lesssim& \sup_{0 \leq y \leq 1+\mu} \int_0^{\frac{3}{4}y}  \frac{1}{\sqrt{\nu (t-s)}} e^{-\frac{(y-z)^2}{16\nu(t-s)}} dz.
\end{aligned}
\end{equation*}
Letting $\frac{y-z}{\sqrt{t-s}}=u$, we obtain
\begin{equation*}
\begin{aligned}
\sup_{0 \leq y \leq 1+\mu} \int_0^{\frac{3}{4}y} \frac{w_t(y)}{w_s(z)} \frac{1}{\sqrt{\nu (t-s)}} e^{-\frac{(y-z)^2}{8\nu(t-s)}} dz\lesssim& \int^{\frac{y}{\sqrt{t-s}}}_{\frac{1-\frac{3}{4}}{\sqrt{t-s}}y} e^{-\frac{u^2}{16\nu}} \frac{1}{\sqrt{\nu}} du\\
\lesssim& \frac{3}{4} \frac{y}{\sqrt{t-s}}  e^{-\frac{1}{16\nu}(\frac{1-\frac{3}{4}}{\sqrt{t-s}})^2 y^2} \frac{1}{\sqrt{\nu}}
\lesssim1.
\end{aligned}
\end{equation*}
Therefore, we have
\begin{equation*}
\sup_{0 \leq y \leq 1+\mu} \int_0^{\frac{3}{4}y} \frac{w_t(y)}{w_s(z)}  \frac{1}{\sqrt{\nu (t-s)}} e^{-\frac{(y-z)^2}{8\nu(t-s)}} dz \lesssim 1.
\end{equation*}

\item When $s<\nu^2/2\gamma^2$, there is $w_s(z) \geq 1/2$ and thus
\begin{equation*}
\frac{w_t(y)}{w_s(z)} \lesssim 1.
\end{equation*}
Therefore, it holds that
\begin{equation*}
\begin{aligned}
\sup_{0 \leq y \leq 1+\mu} \int_0^{\frac{3}{4}y} \frac{w_t(y)}{w_s(z)} \frac{1}{\sqrt{\nu (t-s)}} e^{-\frac{(y-z)^2}{8 \nu (t-s)}} dz
\lesssim&\sup_{0 \leq y \leq 1+\mu} \int_0^{\frac{3}{4}y}  \frac{1}{\sqrt{\nu (t-s)}} e^{-\frac{(y-z)^2}{8 \nu (t-s)}} dz\\
\lesssim& \sup_{0 \leq y \leq 1+\mu} \frac{3}{4} \frac{y}{\sqrt{\nu (t-s)}} e^{-\frac{y^2}{128 \nu (t-s)}} dz \lesssim 1.
\end{aligned}
\end{equation*}
\end{enumerate}
In summary, we have
\begin{equation*}
\Vert e^{\e_0(1+\mu-y)_+|\xi|}(I_1+ I_2) \Vert_{\mathcal{L}^{\infty}_{\mu, \nu, t}} \lesssim \Vert  e^{\e_0(1+\mu-y)_+|\xi|} N_{\xi}(s)\Vert_{\mathcal{L}^{\infty}_{\mu, \nu, s}}.
\end{equation*}
\textbf{(3)  Estimate of $\sup_{y \in \Omega_{\mu}} |e^{\e_0(1+\mu-y)_{+}|\xi|} w_t(y) I_3|$:} \\
Recalling the definition of $I_3$, we have
\begin{equation*}
\begin{aligned}
&\sup_{y \in \Omega_{\mu}} |e^{\e_0(1+\mu-y)_{+}|\xi|} w_t(y) I_3|\\
&\quad =\sup_{y \in \Omega_{\mu}} \left| \int_{\frac{1}{2}y}^{1+\mu} \frac{w_t(y)}{w_s(z)} \frac{y}{z} \left(1- \phi\left(\frac{z}{y}\right)\right) H_{\xi}(t-s,y,z) e^{\e_0(z-y)_+|\xi|}
 e^{\e_0(1+\mu-z)_{+}|\xi|}\partial_z N_{\xi}(s,z)z w_s(z)dz \right|.
\end{aligned}
\end{equation*}
In this case, there is $z> \frac{1}{2}y$ which implies $y/z < 2$, and thus 
\begin{equation*}
w(y\frac{\nu^{2\beta}}{t^\beta \gamma^{2\beta}}) \lesssim w(2z \frac{\nu^{2\beta}}{s^\beta \gamma^{2\beta}}) \lesssim w(z \frac{\nu^{2\beta}}{s^\beta \gamma^{2\beta}})
\end{equation*}
holds.
Therefore, we obtain
\begin{equation*}
\begin{aligned}
&\sup_{y \in \Omega_{\mu}} |e^{\e_0(1+\mu-y)_{+}|\xi|} w_t(y) I_3|\\
\lesssim&\sup_{y \in \Omega_{\mu}}  \int_{\frac{1}{2}y}^{1+\mu} \frac{w_t(y)}{w_s(z)}  (1- \phi(\frac{z}{y})) H_{\xi}(t-s,y,z) e^{\e_0(z-y)_+|\xi|}dz\\
&\ \ \ \ \ \ \ \ \ \ \ \ \ \ \ \ \times \Vert e^{\e_0(1+\mu-z)_{+}|\xi|}z\partial_z N_{\xi}(s,z) \Vert_{\mathcal{L}^{\infty}_{\mu, \nu,s}}\\
\lesssim&\sup_{y \in \Omega_{\mu}} \int_{\frac{1}{2}y}^{1+\mu} \frac{w_t(y)}{w_s(z)} H_{\xi}(t-s,y,z) e^{\e_0(z-y)_+|\xi|}dz \Vert e^{\e_0(1+\mu-z)_{+}|\xi|}z\partial_z N_{\xi}(s,z) \Vert_{\mathcal{L}^{\infty}_{\mu, \nu,s}}.
\end{aligned}
\end{equation*}
We are done with the $I_3$ term once we have the estimate of the quantity 
\begin{equation}\label{target2}
 \left\Vert \frac{w_t(y)}{w_s(z)} H_{\xi}(t-s,y,z) e^{\e_0(z-y)_+|\xi|}\right\Vert_{L^{\infty}_y L^1_z}.
\end{equation}
Note that
\begin{equation*}
\frac{w_t(y)}{w_s(z)} \lesssim \frac{w(y \frac{\nu^{2\beta}}{t^\beta \gamma^{2\beta}})}{w(z \frac{\nu^{2\beta}}{s^\beta \gamma^{2\beta}} )} + \frac{(1-\frac{t\gamma^2}{\nu^2})_+}{(1-\frac{s \gamma^2}{\nu^2})_+} \lesssim 1
\end{equation*}
and
\begin{equation*}
e^{\e_0(z-y)_+|\xi|} \leq e^{\frac{\e_0(z-y)^2}{2\nu (t-s)}} e^{\frac{1}{2}|\xi|^2 \nu (t-s)}.
\end{equation*}
Hence, if we choose $\e_0$ small enough, there is 
\begin{equation}\label{eq2}
\begin{aligned}
 &\int_{\frac{1}{2}y}^{1+\mu} \frac{w_t(y)}{w_s(z)} H_{\xi}(t-s,y,z) e^{\e_0(z-y)_+|\xi|}dz \\
\lesssim&\int_0^{\infty}  e^{\frac{\e_0(z-y)^2}{2\nu (t-s)}} e^{\frac{1}{2}|\xi|^2 \nu (t-s)} \frac{1}{\sqrt{\nu (t-s)}} e^{-\frac{(y-z)^2}{4\nu (t-s)}} e^{-\nu\xi^2(t-s)} dz\lesssim1,
\end{aligned}
\end{equation}
from where we further arrive at
\begin{equation*}
\left \Vert e^{\e_0(1+ \mu -y)_+|\xi|} I_3 \right\Vert_{\mathcal{L}^{\infty}_{\mu, \nu, t}} \lesssim  \left\Vert e^{(1+\mu-z)_+|\xi|} z\partial_z N_{\xi}(s,z)   \right\Vert_{\mathcal{L}^{\infty}_{\mu, \nu, s}}.
\end{equation*}
\textbf{(4) Estimate of $\sup_{y \in \Omega_{\mu}} |e^{\e_0(1+\mu-y)_{+}|\xi|} w_t(y) I_4|$:}\\
Recalling the definition of $I_4$ and splitting the integral interval into $[1+\mu, 1+ \tilde{\mu}]$ and $[1+ \tilde{\mu}, \infty)$, we have
\begin{equation*}
\begin{aligned}
&\sup_{y \in \Omega_\mu}|e^{\e_0(1+\mu-y)_{+}|\xi|} w_t(y) I_4|\\
=&\sup_{y \in \Omega_\mu}\Big|e^{\e_0(1+\mu-y)_{+}|\xi|} w_t(y) y \int_{1+\mu}^{\infty} H_{\xi} (t-s,y,z) \partial_z N_{\xi}(s,z) dz\Big|\\
\leq&\sup_{y \in \Omega_\mu}\Big|e^{\e_0(1+\mu-y)_{+}|\xi|} w_t(y) y \int_{1+\mu}^{1+\tilde{\mu}} H_{\xi} (t-s,y,z) \partial_z N_{\xi}(s,z) dz\Big|\\
& + \sup_{y \in \Omega_\mu}\Big|e^{\e_0(1+\mu-y)_{+}|\xi|} w_t(y) y \int_{1+\tilde{\mu}}^{\infty} H_{\xi} (t-s,y,z) \partial_z N_{\xi}(s,z) dz\Big|\\
\lesssim& \sup_{y \in \Omega_\mu} \Big|\int_{1+\mu}^{1+\tilde{\mu}} e^{\e_0(z-y)_+|\xi|} H_{\xi}(t-s,y,z) \frac{w_t(y)}{w_s(z)} \frac{y}{z} e^{\e_0{(1+\mu -z)_+|\xi|}}w_s(z) z \partial_z N_{\xi}(s,z) dz\Big|\\
&+ \sup_{y \in \Omega_\mu}\Big| \int_{1+\tilde{\mu}}^{\infty} H_{\xi} (t-s,y,z) e^{\e_0(z-y)_{+}|\xi|} w_t(y)y e^{\e_0(1+\mu-z)_{+}|\xi|} \partial_z N_{\xi}(s,z) dz\Big|.
\end{aligned}
\end{equation*}
For the first integral, noticing that $z>1+\mu>y$, we obtain
\begin{equation*}
\begin{aligned}
&\sup_{y \in \Omega_\mu} \Big|\int_{1+\mu}^{1+\tilde{\mu}} e^{\e_0(z-y)_+|\xi|} H_{\xi}(t-s,y,z) \frac{w_t(y)}{w_s(z)} \frac{y}{z} e^{\e_0{(1+\mu -z)_+|\xi|}}w_s(z) z \partial_z N_{\xi}(s,z) dz \Big|\\
\lesssim& \sup_{y \in \Omega_\mu} \int_{1+\mu}^{1+\tilde{\mu}} e^{\e_0(z-y)_+|\xi|} H_{\xi}(t-s,y,z) \frac{w_t(y)}{w_s(z)} dz \Big\Vert e^{\e_0{(1+\mu -z)_+|\xi|}} z \partial_z N_{\xi}(s,z) \Big\Vert_{\mathcal{L}^{\infty}_{\tilde{\mu}, \nu, s}}.
\end{aligned}
\end{equation*}
While for the second integral, using $w_t(y)y \lesssim1 $ and  $z> 1+\tilde{\mu} > 1+\mu$, we get
\begin{equation*}
\begin{aligned}
 &\sup_{y \in \Omega_\mu}\Big| \int_{1+\tilde{\mu}}^{\infty} H_{\xi} (t-s,y,z) e^{\e_0(z-y)_{+}|\xi|} w_t(y)y e^{\e_0(1+\mu-z)_{+}|\xi|} \partial_z N_{\xi}(s,z) dz\Big|\\
\lesssim & \sup_{y \in \Omega_\mu}\Big| \int_{1+\tilde{\mu}}^{\infty} H_{\xi} (t-s,y,z) e^{\e_0(z-y)_{+}|\xi|}\partial_z N_{\xi}(s,z) dz\Big|\\
\lesssim&  \sup_{y \in \Omega_\mu} \Big(\int_{1+\tilde{\mu}}^{\infty} |H_{\xi} (t-s,y,z) e^{\e_0(z-y)_{+}|\xi|}|^2 dz\Big)^{1/2} \Big(\int_{1+\tilde{\mu}}^{\infty} |\partial_z N_{\xi}(s,z)|^2 dz|\Big)^{1/2}.
\end{aligned}
\end{equation*}
Therefore, in order to bound the $I_4$ term, we are left to estimate the following two quantities:
\begin{equation}\label{target3}
\Big\Vert \mathbbm{1}_{\{0 \leq y \leq 1+\mu\}}  \mathbbm{1}_{\{1+\mu \leq z \leq 1+\tilde{\mu}\}} e^{\e_0(z-y)_+|\xi|} H_{\xi}(t-s, y,z) \frac{w_t(y)}{w_s(z)} \Big\Vert_{L^{\infty}_yL^{1}_z},
\end{equation}
and
\begin{equation}\label{target4}
\Big\Vert \mathbbm{1}_{\{0 \leq y \leq 1+\mu\}} \mathbbm{1}_{\{z \geq 1+\tilde{\mu}\}}  H_{\xi} (t-s,y,z) e^{\e_0(z-y)_{+}|\xi|} \Big\Vert_{L^{\infty}_y L^2_z}.
\end{equation}
Since $z \geq 1+\mu \geq y$ and $t>s$, there is $y\frac{\nu^{2\beta}}{t^{\beta}} \leq z\frac{\nu^{2\beta}}{s^{\beta}} $.
Hence, we have
\begin{equation*}
\frac{w_t(y)}{w_s(z)} \leq \frac{w\left(y\frac{\nu^{2\beta}}{t^{\beta} \gamma^{2\beta}}\right)}{w\left(z\frac{\nu^{2\beta}}{s^{\beta} \gamma^{2\beta}}\right)} + \frac{\left(1- \frac{\gamma^2 t}{\nu^2}\right)_+}{\left(1- \frac{\gamma^2 s}{\nu^2}\right)_+} \lesssim 1.
\end{equation*}
Following the same steps as in~\eqref{eq2}, the quantity defined in~\eqref{target3} is bounded by
\begin{equation*}
\begin{aligned}
& \sup_{y \in \Omega_\mu} \Big|\int_{1+\mu}^{1+\tilde{\mu}} e^{\e_0(z-y)_+|\xi|} H_{\xi}(t-s,y,z) \frac{w_t(y)}{w_s(z)} dz \Big|\\
\lesssim& \sup_{y \in \Omega_\mu} \Big|\int_{1+\mu}^{1+\tilde{\mu}} e^{\e_0(z-y)_+|\xi|} H_{\xi}(t-s,y,z) dz\Big| \lesssim 1.\\
\end{aligned}
\end{equation*}
By similar arguments, for $\e_0$ small enough, we obtain
\begin{equation*}
\begin{aligned}
&\int_{1+\tilde{\mu}}^{\infty} \Big|H_{\xi} (t-s,y,z) e^{\e_0(z-y)_{+}|\xi|}\Big|^2 dz\\
\lesssim& \int_{1+\tilde{\mu}}^{\infty} \Big|e^{\frac{\e_0^2(z-y)^2}{2\nu (t-s)}} e^{\frac{1}{2} \nu (t-s) |\xi|^2} \frac{1}{\sqrt{\nu (t-s)}} e^{-\frac{(y-z)^2}{4\nu(t-s)}} e^{-\nu \xi^2(t-s)}\Big|^2 dz
\lesssim \int_{1+\tilde{\mu}}^{\infty} \Big|\frac{1}{\sqrt{\nu (t-s)}} e^{-\frac{(y-z)^2}{16\nu(t-s)}}\Big|^2 dz\\
\lesssim& \int_0^{\infty} \frac{1}{\sqrt{\nu (t-s)}} e^{-\frac{(y-z)^2}{16\nu(t-s)}} dz  \frac{\tilde{\mu}-\mu}{\sqrt{\nu (t-s)}} e^{-\frac{(\tilde{\mu}-\mu)^2}{16\nu(t-s)}} \frac{1}{\tilde{\mu}- \mu}\\
\lesssim&  \frac{1}{\tilde{\mu}- \mu}.\\
\end{aligned}
\end{equation*}
Therefore, we have
\begin{equation*}
\begin{aligned}
&\sup_{y \in \Omega_\mu}|e^{\e_0(1+\mu-y)_{+}|\xi|} w_t(y) I_4|\\
\lesssim& \Vert e^{\e_0{(1+\mu -z)_+|\xi|}} z \partial_z N_{\xi}(s,z) \Vert_{\mathcal{L}^{\infty}_{\tilde{\mu}, \nu, s}} + \frac{1}{\sqrt{\tilde{\mu}-\mu}}\Vert \partial_z N_{\xi}(z) \Vert_{L^2(z \geq 1+ \tilde{\mu})},
\end{aligned}
\end{equation*}
which completes the proof.
\end{proof}

\begin{Lemma}
Assume that $\mu$ and $\tilde{\mu}$ satisfy $0< \mu < \tilde{\mu} < \mu_0 - \gamma \sqrt{s}$ with $\tilde{\mu}- \mu \geq \frac{1}{c} (\mu_0-\mu - \gamma \sqrt{s} )$ for some constant $c \geq 1$. Then we have
\begin{equation}
\begin{aligned}
&\Big\Vert e^{\e_0(1+ \mu -y)_{+}|\xi|} \xi^i(y \partial y)^j \int_0^{\infty} H_{\xi}(t-s,y,z) N_{\xi}(s,z) dz \Big\Vert_{\mathcal{L}_{\mu, \nu, t}^{\infty}}\\
 \lesssim& \Vert e^{\e_0(1+\mu-y)_{+}|\xi|}  N_{\xi}(s,z) \Vert_{\mathcal{L}_{\tilde{\mu},\nu,s}^{\infty}} +  \Vert e^{\e_0(1+\mu-y)_{+}|\xi|} (y \partial_y)(\partial_x^iN)_{\xi}(s,z) \Vert_{\mathcal{L}_{\tilde{\mu},\nu,s}^{\infty}}\\& + \frac{1}{\sqrt{\mu_0-\mu -\gamma \sqrt{s}}} \Vert  \partial_y^j (\partial_x^iN)_{\xi}(s,z) \Vert_{L^2(z \geq 1+ \tilde{\mu})},
\end{aligned}
\end{equation}
for all $i+j \leq 1$.
\end{Lemma}
\begin{proof}
In the previous lemma, we have proved the result in the case $(i,j)=(0,1)$.  The proof of case $(i,j)=(0,0)$ is similar to that of $(i,j)=(0,1)$. The result when $(i,j)=(1,0)$ follows from the bound of $(i,j)=(0,0)$ replacing $N$ by $\partial_xN$.
\end{proof}

By summing up the above estimates in $\xi$, we have the following corollary.
\begin{Corollary}\label{X1}
Assume that $\mu$ and $\tilde{\mu}$ satisfies $0< \mu < \tilde{\mu} < \mu_0 - \gamma \sqrt{s}$ with $\tilde{\mu}- \mu \geq \frac{1}{c} (\mu_0-\mu - \gamma \sqrt{s} )$ for some constant $c \geq 1$. Then we have
\begin{equation}
\begin{aligned}
&\Big\Vert\xi^i(y \partial y)^j \int_0^{\infty}{H}(t-s,y,z) N_{\xi}(s,z) dz \Big\Vert_{X_{\mu,t}}\\ 
\lesssim& \Vert \partial_x^i N(s,z) \Vert_{X_{\tilde{\mu},s}} +  \Vert (\partial_x^i (y \partial_y)^jN)_{\xi}(s,z) \Vert_{X_{\tilde{\mu},s}} \\
&\ \ \ \ \ \ \ \ + \frac{1}{\sqrt{\mu_0-\mu -\gamma \sqrt{s}}} \sum_{\xi \in \mathbb{Z}}\Vert \partial_y^j (\partial_x^iN)_{\xi}(s,z) \Vert_{L^2(y \geq 1+ \tilde{\mu})},
\end{aligned}
\end{equation}
 for all $i+j \leq 1$ and $s \in [0,t]$.
\end{Corollary}

\begin{Lemma}\label{l2}
Assume that $\mu$ and $\tilde{\mu}$ satisfies $0< \mu < \tilde{\mu} < \mu_0 - \gamma \sqrt{s}$ with $\tilde{\mu}- \mu \geq \frac{1}{c} (\mu_0-\mu - \gamma \sqrt{s} )$ for some constant $c \geq 1$. Then we have
\begin{equation}
	\label{resi:kern}
\begin{aligned}
&\Big\Vert e^{\e_0(1+\mu-y)_+|\xi|} \xi^i (y\partial y)^j \int_0^{\infty} R_{\xi} (t-s,y,z) N_{\xi}(s,z) dz \Big\Vert_{\mathcal{L}^{\infty}_{\mu,\nu,t}} \\
\lesssim& \Vert e^{\e_0(1+ \mu-y)_+|\xi|} \xi^i (y \partial_y)^j N_{\xi}(s,y) \Vert_{\mathcal{L}^{\infty}_{\tilde{\mu}, \nu,s}} + \Vert e^{\e_0(1+ \mu-y)_+|\xi|}  N_{\xi}(s,y) \Vert_{\mathcal{L}^{\infty}_{\tilde{\mu}, \nu,s}}  \\
&\ \ \ \ \ \ +\frac{1}{\sqrt{\mu_0- \mu-\gamma \sqrt{s}}} \Vert  \xi^i \partial_y^j N_{\xi}(s,y)\Vert_{L^2(y \geq 1+ \tilde{\mu})},
\end{aligned}
\end{equation}
 for all $i+j \leq 1$ and $s \in [0,t]$.
\end{Lemma}
\begin{proof}
By the same argument as in the previous lemma, we only need to prove the result for the case $(i,j)=(0,1)$. In fact, we show that the quantities \eqref{target1}, \eqref{target2}, \eqref{target3}, and \eqref{target4} share similar bounds as in Lemma~\ref{l1} with kernel $H_{\xi}(t-s,y,z)$ replaced by $R_{\xi}(t-s,y,z)$. To be more specific, we are going to estimate the following term,
\begin{equation*}
\begin{aligned}
\frac{w_t(y)}{w_s(z)} \left( |R_{\xi}(t-s,y,z) + |y \partial_y R_{\xi}(t-s,y,z)|\right)
\lesssim \frac{w_t(y)}{w_s(z)}  \left(b e^{-\frac{\theta_0}{2}b(y+z)} + \frac{1}{\sqrt{\nu t}}e^{-\frac{\theta_0}{2} \frac{(y+z)^2}{\nu t}} e^{-\frac{\nu \xi^2 t}{8}}\right),
\end{aligned}
\end{equation*}
which requires the estimate of the quantity in~\eqref{target1} with $H_{\xi}(t-s,y,z)$ replaced by $R_{\xi}(t-s,y,z)$. We consider two cases.
\textbf{(1) When $s> \nu^2/2\gamma^2$}, we have
\begin{equation*}
\begin{aligned}
\frac{w_t(y)}{w_s(z)} \lesssim& \frac{w\left(y\frac{\nu^{2\beta}}{t^{\beta} \gamma^{2\beta}}\right)}{w\left(z \frac{\nu^{2\beta}}{s^{\beta} \gamma^{2\beta}}\right)} + 1
\lesssim \frac{1}{\sqrt{\nu}} \left(\sqrt{\nu}\mathbbm{1}_{\{0 \leq y\frac{\nu^{2\beta}}{t^\beta \gamma^{2\beta}} \leq \nu\}} + y\frac{\nu^{2\beta}}{t^\beta\gamma^{2 \beta}} \mathbbm{1}_{\{y\frac{\nu^{2\beta}}{t^\beta \gamma^{2\beta}} > \nu\}}\right)   +1\\
\lesssim& \nu^{2\beta-1/2}\frac{y}{s^{\beta}}+1 \lesssim \frac{y}{\sqrt{\nu}} +1,
\end{aligned}
\end{equation*}
and
\begin{equation*}
\frac{w_t(y)}{w_s(z)} \lesssim e^{y \frac{\nu^{2\beta}}{C \gamma^{2\beta} \sqrt{\nu} t^\beta}} \lesssim  e^{y \frac{\nu^{2\beta}}{C \gamma_0^{2\beta} \sqrt{\nu} t^\beta}}.
\end{equation*}
Now~\eqref{target1} is bounded by
\begin{equation*}
\begin{aligned}
&\Big\Vert \mathbbm{1}_{0 \leq 1 \leq 1+\mu} \mathbbm{1}_{\{0 \leq z \leq \frac{3}{4} y\}}\frac{w_t(y)}{w_s(z)} (|R_{\xi}(t-s,y,z) + |y \partial_y R_{\xi}(t-s,y,z)|)\Big\Vert_{L^{\infty}_yL^1_z}\\
\lesssim& \sup_{0 \leq y \leq 1+\mu} \int_0^{\frac{3}{4}y} \left(be^{-\frac{\theta_0}{2}b(y+z)} +\frac{1}{\sqrt{\nu t}} e^{-\frac{\theta_0}{2} \frac{(y+z)^2}{\nu t}} e^{-\frac{\nu \xi^2 t}{8}}\right) \frac{w_t(y)}{w_s(z)} dz\\
\lesssim& \sup_{0 \leq y \leq 1+\mu} \int_0^{\frac{3}{4}y} be^{-\frac{\theta_0}{2} b(y+z)} \frac{y}{\sqrt{\nu}} dz \\
&+ \sup_{0 \leq y \leq 1+\mu} \int_0^{\frac{3}{4}y} \frac{1}{\sqrt{\nu t}} e^{-\frac{\theta_0}{2} \frac{(y+z)^2}{\nu t}} e^{-\frac{\nu \xi^2 t}{8}}  e^{y \frac{\nu^{2\beta}}{C \gamma_0^{2\beta} \sqrt{\nu} t^\beta}} dz\\
&+\sup_{0 \leq y \leq 1+\mu} \int_0^{\frac{3}{4}y} \left(be^{-\frac{\theta_0}{2}b(y+z)} +\frac{1}{\sqrt{\nu t}} e^{-\frac{\theta_0}{2} \frac{(y+z)^2}{\nu t}} e^{-\frac{\nu \xi^2 t}{8}}\right)  dz\\
:=&I_1 +I_2 + I_3.
\end{aligned}
\end{equation*}
For the term $I_1$, there is
\begin{equation*}
\begin{aligned}
I_1 \lesssim& \sup_{0 \leq y \leq 1+\mu}\int_0^{\frac{3}{4}y}  b e^{-\frac{\theta_0}{2}bz} dz \cdot e^{-\frac{\theta_0}{2}by} \frac{y}{\sqrt{\nu}}
\lesssim \sup_{0 \leq y \leq 1+\mu}by \cdot e^{-\frac{\theta_0}{2}by} \cdot \frac{y}{\sqrt{\nu}}\\
\lesssim& \sup_{0 \leq y \leq 1+\mu} \left(\frac{1}{\sqrt{\nu}} + |\xi|\right) e^{-\frac{\theta_0}{2}\frac{y}{\sqrt{\nu}}} \cdot \frac{y^2}{\sqrt{\nu}}\cdot e^{-\frac{\theta_0}{2} |\xi|y}\\
\lesssim& \sup_{0 \leq y \leq 1+\mu}    e^{-\frac{\theta_0}{2}\frac{y}{\sqrt{\nu}}} \cdot \frac{y^2}{{\nu}} +  e^{-\frac{\theta_0}{2} \frac{y}{\sqrt{\nu}}} \cdot \frac{y}{\sqrt{\nu}}  \cdot y |\xi| e^{-\frac{\theta_0}{2} |\xi| y} \lesssim 1.
\end{aligned}
\end{equation*}
While for the second piece $I_2$, we get
\begin{equation*}
\begin{aligned}
I_2 \lesssim& \sup_{0 \leq y \leq 1+\mu} \int_0^{\frac{3s^\beta}{4t^\beta}y} \frac{t^\beta}{s^\beta} \frac{1}{\sqrt{\nu t}} e^{-\theta_0 \frac{y^2+z^2}{\nu t}} dz \cdot e^{-\frac{\nu \xi^2 t}{8}} \cdot e^{\frac{y \nu^{2\beta}}{C \gamma_0^{2\beta} \sqrt{\nu} t^\beta}}\\
\lesssim& \sup_{0 \leq y \leq 1+\mu}  \frac{y}{\sqrt{\nu t}} e^{-\theta_0 \frac{y^2+z^2}{\nu t}} \cdot e^{\frac{y }{C\gamma_0^{2\beta}\sqrt{\nu} t^\beta}}\\
\lesssim& \sup_{0 \leq y \leq 1+\mu} \frac{y}{\sqrt{\nu t}} e^{-\sqrt{\theta_0} \frac{y}{\sqrt{\nu t}}} \cdot e^{\frac{y }{C \gamma_0^{2\beta} \sqrt{\nu t} } t^{1/2-\beta}}.
\end{aligned}
\end{equation*}
By choosing $C$ large enough, there is  
\begin{equation*} 
I_2 \lesssim \sup_{0 \leq y \leq 1+\mu} \frac{y}{\sqrt{\nu t}} e^{-\frac{\sqrt{\theta_0}}{2} \frac{y}{\sqrt{\nu t}}} \lesssim1.
\end{equation*}
For the term $I_3$, we have
\begin{equation*}
\begin{aligned}
I_3 \lesssim& \sup_{0 \leq y \leq 1+\mu} \int_0^{\frac{3s^\beta}{4t^\beta} y} \frac{t^\beta}{s^\beta} \left(b e^{-\frac{\theta_0}{2}b(y+z)} + \frac{1}{\sqrt{\nu t}}e^{-\frac{\theta_0}{2} \frac{(y+z)^2}{\nu t}} e^{-\frac{\nu \xi^2 t}{8}}\right) dz\\
\lesssim& \sup_{0 \leq y \leq 1+\mu} \int_0^{\frac{3s^\beta}{4t^\beta} y} \frac{t^\beta}{s^\beta} b e^{-\frac{\theta_0}{2}bz}dz \cdot e^{-\frac{\theta_0}{2}by} +  \sup_{0 \leq y \leq 1+\mu} \int_0^{\frac{3s^\beta}{4t^\beta} y} \frac{t^\beta}{s^\beta}\frac{1}{\sqrt{\nu t}}e^{-\frac{\theta_0}{2} \frac{y^2+z^2}{\nu t}}dz\\
\lesssim& \sup_{0 \leq y \leq 1+\mu} by e^{-\frac{\theta_0}{2}by} + \frac{y}{\sqrt{\nu t}}e^{-\frac{\theta_0}{2} \frac{y^2}{\nu t}} \lesssim 1.
\end{aligned}
\end{equation*}
\textbf{(2)When $s < \nu^2/2\gamma^2$}, there is
\begin{equation*}
\frac{w_t(y)}{w_s(z)} \lesssim 1,
\end{equation*}
and thus the estimate of quantity \eqref{target1} is the same as that of $I_3$.\\
The rest of the proof follows exactly as the proof of Lemma 4.3 in \cite{KukVicWan20}.
\end{proof}
Summing up the estimate~\eqref{resi:kern} in $\xi$, we have the following corollary.
\begin{Corollary}\label{X2}
Assume that $\mu$ and $\tilde{\mu}$ satisfies $0< \mu < \tilde{\mu} < \mu_0 - \gamma \sqrt{s}$ with $\tilde{\mu}- \mu \geq \frac{1}{c} (\mu_0-\mu - \gamma \sqrt{s} )$ for some constant $c \geq 1$. Then
\begin{equation}
\begin{aligned}
\left\Vert \partial_x^i (y \partial_y)^j \int_0^{\infty} R(t-s,y,z) N(s,z) dz \right\Vert_{X_{\mu , t}} \lesssim&  \Vert \partial_x^i (y\partial_y)^j N(s) \Vert_{X_{\tilde{\mu},s}} + \Vert N(s) \Vert_{X_{\tilde{\mu},s}}\\
&+\frac{1}{\sqrt{\mu_0-\mu-\gamma \sqrt{s}}} \sum_{\xi} \Vert  \partial_x^i (y\partial_y)^j N_{\xi} \Vert_{L^2(y \geq1+\tilde{\mu})},
\end{aligned}
\end{equation}
 for $i+j \leq 1$ and $s \in [0,t]$.
\end{Corollary}
The next lemma deals with contribution from the boundary.
\begin{Lemma}\label{X3}
Let $\mu \in (0, \mu_0 - \gamma s)$ be arbitrary, then 
\begin{equation}
\begin{aligned}
&\Vert \partial_x^i (y \partial_y)^j G(t-s) \partial_z \Delta^{-1}N(s,z)|_{z=0} \Vert_{X_{\mu,t}}\\
 \leq& \frac{1}{\sqrt{t-s}}\left(\frac{1}{\sqrt{\nu}} \mathbbm{1}_{\{0 \leq t \leq {\nu^2}/\gamma^2\}} + \frac{1}{s^\beta}\right) (\Vert  \partial_x^i N(s)\Vert_{Y_{\mu}} + \Vert \partial_x^i N(s) \Vert_{S_\mu})\\
&+ \left(1 + \frac{1}{\sqrt{\nu}} \mathbbm{1}_{\{0 \leq t \leq {\nu^2}/\gamma^2\}}\right) \Vert  \partial_x^i N(s)\Vert_{X_{\mu,s}}, 
\end{aligned}
\end{equation}
for $i+j \leq 1$.
\end{Lemma}
\begin{proof}
For $\xi \in \mathbb{Z}$, the kernel $G_{\xi}(t-s,y,0)$ is the sum of two trace operators:
\begin{equation*}
T_1(t-s,y)= \tilde{H}_{\xi}(t-s,y,0) = \frac{2}{\sqrt{\nu (t-s)}} e^{-\frac{y^2}{4 \nu (t-s)}} e^{-\gamma \xi^2 (t-s)},
\end{equation*}
and
\begin{equation*}
T_2(t-s,y) = R_{\xi}(t-s,y,0).
\end{equation*}
For the operators $T_1$ and $T_2$, we have the following estimates:
\begin{equation*}
y \partial_y T_1(t-s,y) \lesssim \frac{1}{\sqrt{\nu (t-s)}} e^{-\frac{y^2}{8 \nu (t-s)}} e^{-\gamma \xi^2 (t-s)},
\end{equation*}
and
\begin{equation}\label{kernel T2}
|T_2(t-s,y)| + |y \partial_yT_2(t-s,y)| \lesssim b e^{-\frac{\theta_0}{2} by } + \frac{1}{\sqrt{\nu(t-s)}} e^{-\frac{\theta_0}{2} \frac{y^2}{\nu(t-s)}} e^{-\frac{\nu \xi^2 (t-s)}{8}}.
\end{equation}
Moreover, by representation formula, we have
\begin{equation*}
\begin{aligned}
(\partial_z \Delta^{-1}N_{\xi}(s,z))|_{z=0}=& -\int_0^{\infty} e^{-|\xi|z} N_{\xi}(s,z) dz\\
=&-\int_0^{1+\mu} e^{-|\xi|z} N_{\xi}(s,z) dz -\int_{1+\mu}^{\infty} e^{-|\xi|z} N_{\xi}(s,z) dz \\
:=&I_1 +I_2.
\end{aligned}
\end{equation*}
We take care of contribution from the operator $T_1$ first.\\
\textbf{(1) For the $I_1$ part, we have:}
\begin{equation*}
\begin{aligned}
&\sup_{0 \leq y \leq 1+\mu} |e^{\e_0(1+\mu-y)_+|\xi|} w_t(y) T_1(t-s,y) I_1|\\
\lesssim& \sup_{0 \leq y \leq 1+\mu} e^{\e_0(1+\mu-y)_+|\xi|} w_t(y) \frac{2}{\sqrt{\nu (t-s)}} e^{-\frac{y^2}{4 \nu (t-s)}} e^{-\gamma \xi^2 (t-s)} \int_0^{1+\mu} e^{-|\xi|z} N_{\xi}(s,z) dz\\
\lesssim& \sup_{0 \leq y \leq 1+\mu} \frac{w_t(y)}{\sqrt{\nu (t-s)}} e^{-\frac{y^2}{4 \nu (t-s)}} \int_0^{1+\mu} e^{\e_0(z-y)_+|\xi|} e^{-|\xi|z} e^{\e_0(1+\mu-z)_+|\xi|} N_{\xi}(s,z) dz\\
\lesssim& \sup_{0 \leq y \leq 1+\mu}  \frac{w_t(y)}{\sqrt{\nu (t-s)}} e^{-\frac{y^2}{4 \nu (t-s)}}  \int_0^{1+\mu}e^{\e_0(1+\mu-z)_+|\xi|} N_{\xi}(s,z) dz\\
\lesssim& \sup_{0 \leq y \leq 1+\mu} \frac{w\left(y\frac{\nu^{2\beta} }{t^\beta\gamma^{2\beta}}\right) + \left(1-\frac{t \gamma^2}{\nu^2}\right)_+}{\sqrt{\nu (t-s)}} \Vert e^{\e_0(1+\mu-z)_+|\xi|} N_{\xi}(s,z)  \Vert_{\mathcal{L}^1_{\mu}}.
\end{aligned}
\end{equation*}
Noticing that
\begin{equation*}
w\left(y\frac{\nu^{2\beta}}{t^\beta \gamma^{2\beta}}\right) \lesssim \sqrt{\nu} e^{\frac{y \nu^{2\beta}}{C \sqrt{\nu} t^\beta \gamma^{2\beta}}} \lesssim\sqrt{\nu} e^{\frac{y \nu^{2\beta}}{C \gamma_0^{2\beta}\sqrt{\nu} t^\beta }},
\end{equation*}
we have
\begin{equation*}
\begin{aligned}
\frac{1}{\sqrt{\nu(t-s)}} w\left(y\frac{\nu^{2\beta}}{t^\beta \gamma^{2\beta}}\right)e^{-\frac{y^2}{4 \nu (t-s)}}\lesssim& \frac{1}{\sqrt{t-s}} e^{\frac{y}{C \gamma_0^{2\beta} \sqrt{\nu} t^\beta}}e^{-\frac{y^2}{4 \nu (t-s)}}\\
\lesssim& \frac{1}{\sqrt{t-s}} e^{\frac{y}{C\gamma_0^{2\beta} \sqrt{\nu} t^\beta}}e^{-\frac{y^2}{4 \nu t}}\\
\lesssim& \frac{1}{\sqrt{t-s}} e^{\frac{y}{C\gamma_0^{2\beta} \sqrt{\nu} t^\beta}}e^{-\frac{y}{2\sqrt{ \nu t}}}.\\
\end{aligned}
\end{equation*}
Now since $\beta \in (0, 1/2)$, we can pick $C$ large enough so that 
\begin{equation*}
\frac{w\left(y \frac{\nu^{2\beta}}{t^\beta\gamma^{2\beta}}\right) }{\sqrt{\nu(t-s)}} \lesssim \frac{1}{\sqrt{t-s}}.
\end{equation*}
On the other hand, we notice that
\begin{equation*}
\begin{aligned}
\frac{\left(1-\frac{t \gamma^2}{\nu^2}\right)_+}{\sqrt{\nu (t-s)}} \lesssim \mathbbm{1}_{\{0 \leq t \leq \nu^2/\gamma^2\}} \frac{1}{\sqrt{\nu (t-s)}}. 
\end{aligned}
\end{equation*}
Therefore, we obtain
\begin{equation*}
\begin{aligned}
&\sup_{0 \leq y \leq 1+\mu} |e^{\e_0(1+\mu-y)_+|\xi|} w_t(y) T_1(t-s,y) I_1|\\ \lesssim& \frac{1}{\sqrt{t-s}} \left(1+ \frac{\mathbbm{1}_{\{ 0 \leq t \leq {\nu^2}/\gamma^2 \}}}{\sqrt{\nu}}\right) \Vert e^{\e_0(1+\mu-z)_+|\xi|} N_{\xi}(s,z)  \Vert_{\mathcal{L}^1_{\mu}}.
\end{aligned}
\end{equation*}
\textbf{(2) For the $I_2$ part:}
By the definition of the weight function $w_t(y)$ and Cauchy inequality, there is 
\begin{equation*}
\begin{aligned}
&\sup_{0 \leq y \leq 1+\mu} |e^{\e_0(1+\mu-y)_+|\xi|} w_t(y) T_1(t-s,y) I_2|\\
\lesssim& \sup_{0 \leq y \leq 1+\mu} e^{\e_0(1+\mu-y)_+|\xi|} w_t(y) \frac{2}{\sqrt{\nu (t-s)}} e^{-\frac{y^2}{4 \nu (t-s)}} e^{-\gamma \xi^2 (t-s)} \int_{1+\mu}^{\infty} e^{-|\xi|z} N_{\xi}(s,z) dz\\
\lesssim&\sup_{0 \leq y \leq 1+\mu} \int_0^{\infty} e^{\e_0(1+\mu-z)_+|\xi|} e^{e_0(z-y)_+|\xi|} e^{-|\xi|z} N_{\xi}(s,z) dz \cdot w_t(y)  \frac{1}{\sqrt{\nu (t-s)}} e^{-\frac{y^2}{4 \nu (t-s)}} e^{-\gamma \xi^2 (t-s)}\\
\lesssim& \sup_{0 \leq y \leq 1+\mu}\int_0^{\infty}  N_{\xi}(s,z) dz \cdot \left(w\left(y \frac{\nu^{2\beta}}{t^\beta \gamma^{2\beta}}\right) +\left(1- \frac{t \gamma^2}{\nu^2}\right)_+\right)  \frac{1}{\sqrt{\nu (t-s)}} e^{-\frac{y^2}{4 \nu (t-s)}} e^{-\gamma \xi^2 (t-s)}\\
\lesssim&  \sup_{0 \leq y \leq 1+\mu}\Vert z N_{\xi}(s,z) \Vert_{L^2(z \geq 1+\mu)} \left(\sqrt{\nu} e^{\frac{y \nu^{2\beta}}{C \gamma^{2\beta}\sqrt{\nu} t^\beta}} + \left(1-\frac{t \gamma^2}{\nu^2}\right)_+\right)  \frac{1}{\sqrt{\nu (t-s)}} e^{-\frac{y^2}{4 \nu (t-s)}} e^{-\gamma \xi^2 (t-s)}\\
\lesssim&  \sup_{0 \leq y \leq 1+\mu}\Vert z N_{\xi}(s,z) \Vert_{L^2(z \geq 1+\mu)} \left(\frac{1}{\sqrt{t-s}} e^{\frac{y}{C \gamma_0^{2\beta} \sqrt{\nu}t^\beta}}\cdot e^{-\frac{y}{2\sqrt{\nu t}}} + \left(1-\frac{t \gamma^2}{\nu^2}\right)_+ \frac{1}{\sqrt{\nu(t-s)}}\right).
\end{aligned}
\end{equation*}
Since $\beta \in (0, 1/2)$, we can take $C$ large enough so that
\begin{equation*}
\begin{aligned}
&\sup_{0 \leq y \leq 1+\mu} |e^{\e_0(1+\mu-y)_+|\xi|} w_t(y) T_1(t-s,y) I_2| \\
\lesssim& \frac{1}{\sqrt{t-s}}\left(1+ \frac{\mathbbm{1}_{\{ 0 \leq t \leq {\nu^2}/\gamma^2 \}}}{\sqrt{\nu}}\right)\Vert z N_{\xi}(s,z) \Vert_{L^2(z \geq 1+\mu)}.
\end{aligned}
\end{equation*}
Therefore, we obtain
\begin{equation*}
\begin{aligned}
&\sup_{0 \leq y \leq 1+\mu} |e^{e_0(1+\mu-y)_+|\xi|} w_t(y) T_1(t-s,y) \partial_z \Delta^{-1}N_{\xi}(s,z)|_{z=0}|\\
\lesssim&  \frac{1}{\sqrt{t-s}}\left(1+ \frac{\mathbbm{1}_{\{ 0 \leq t \leq {\nu^2}/\gamma^2 \}}}{\sqrt{\nu}}\right)\\
&\ \ \ \ \ \ \ \ \ \ \ \ \ \ \ \ \ \ \ \ \ \ \times\left(\Vert z N_{\xi}(s,z) \Vert_{L^2(z \geq 1+\mu)} + \Vert e^{\e_0(1+\mu-z)_+|\xi|} N_{\xi}(s,z)  \Vert_{\mathcal{L}^1_{\mu}}\right).
\end{aligned}
\end{equation*}
We now take care of the $T_2$ contribution.\\
\textbf{(1) For the $I_1$ part, we have:}
\begin{equation*}
\begin{aligned}
&\sup_{0 \leq y \leq 1+\mu} |e^{\e_0(1+\mu-y)_+|\xi|} w_t(y) b e^{-\frac{1}{2}\theta_0 by } I_1|\\
\lesssim&\sup_{0 \leq y \leq 1+\mu} \left|e^{\e_0(1+\mu-y)_+|\xi|} \left(w\left(y\frac{\nu{2^\beta}}{t^\beta \gamma^{2\beta}}\right) + \left(1- \frac{t \gamma^2}{\nu^2}\right)_+\right)\right|\\
&\ \ \ \ \ \ \ \ \ \ \ \ \ \ \ \ \ \ \ \times \left|b e^{-\frac{1}{2}\theta_0 by } \int_0^{1+\mu} e^{-|\xi|z} N_{\xi}(s,z) dz\right|\\
\lesssim&\sup_{0 \leq y \leq 1+\mu} \left| \left(w\left(y\frac{\nu^{2\beta}}{t^\beta \gamma^{2\beta}}\right) + \left(1- \frac{t \gamma^2}{\nu^2}\right)_+\right) b e^{-\frac{1}{2}\theta_0 by }\right|\\
&\ \ \ \ \ \ \ \ \ \ \ \ \ \ \ \ \ \ \ \ \times  \left|\int_0^{1+\mu} e^{\e_0(1+\mu-z)_+|\xi|}e^{\e_0(z-y)_+|\xi|}e^{-|\xi|z} N_{\xi}(s,z) dz \right|.\\
\end{aligned}
\end{equation*}
Recalling that for $s<t$, there is
\begin{equation}\label{eq3}
w\left(y\frac{\nu^{2\beta}}{t^\beta \gamma^{2\beta}}\right) \lesssim \sqrt{\nu} + \frac{1}{\gamma_0^{2\beta}}\frac{y}{s^\beta} \nu^{2\beta} \lesssim \sqrt{\nu} + \frac{y}{s^\beta}.
\end{equation}
Hence, we have
\begin{equation*}
\begin{aligned}
&\sup_{0 \leq y \leq 1+\mu} |e^{e_0(1+\mu-y)_+|\xi|} w_t(y) T_2(t-s,y) \partial_z \Delta^{-1}N_{\xi}(s,z)|_{z=0}|\\
\lesssim&\sup_{0\leq y \leq 1+\mu}\left(\sqrt{\nu} + \frac{y}{s^\beta} + \mathbbm{1}_{\{ 0 \leq t \leq {\nu^2}/\gamma^2 \}}\right) be^{-\frac{\theta_0}{2} by}\\
& \ \ \ \ \ \ \ \ \ \ \ \ \ \ \ \ \ \ \ \ \ \ \times \int_0^{1+\mu} e^{\e_0(1+\mu-z)_+|\xi|} e^{\e_0(z-y)_+|\xi|} e^{-|\xi|z} N_{\xi}(s,z) dz\\
\lesssim& \sup_{0 \leq y \leq 1+\mu} \left(\sqrt{\nu} + \frac{y}{s^\beta} + \mathbbm{1}_{\{ 0 \leq t \leq {\nu^2}/\gamma^2 \}}\right) be^{-\frac{\theta_0}{2} by}\\
&\ \ \ \ \ \ \ \ \ \ \ \ \ \ \ \ \ \ \ \ \ \ \ \ \ \times \int_0^{1+\mu} e^{\e_0(1+\mu-z)_+|\xi|} N_{\xi}(s,z) e^{-|\xi|z}dz.
\end{aligned}
\end{equation*}
Recalling that $b= |\xi| + \frac{1}{\sqrt{\nu}}$, we have
\begin{equation*}
\begin{aligned}
&\sup_{0 \leq y \leq 1+\mu} |e^{e_0(1+\mu-y)_+|\xi|} w_t(y) T_2(t-s,y) \partial_z \Delta^{-1}N_{\xi}(s,z)|_{z=0}|\\
\lesssim& \sup_{0 \leq y \leq 1+\mu} \left(1 + \frac{y}{\sqrt{\nu} s^\beta}\right) e^{-\frac{\theta_0}{2} by} \int_0^{1+\mu} e^{\e_0(1+\mu-z)_+|\xi|} N_{\xi}(s,z)e^{-\frac{1}{2}|\xi|z} dz\\
&+  \left(\sqrt{\nu} + \frac{y}{s^\beta}\right)  \int_0^{1+\mu} e^{\e_0(1+\mu-z)_+|\xi|} N_{\xi}(s,z) |\xi| e^{-\frac{1}{2}|\xi|z}dz\\
&+\mathbbm{1}_{\{ 0 \leq t \leq {\nu^2}/\gamma^2 \}}  \frac{1}{\sqrt{\nu}}e^{-\frac{\theta_0}{2} by} \int_0^{1+\mu} e^{\e_0(1+\mu-z)_+|\xi|} N_{\xi}(s,z) e^{-\frac{1}{2}|\xi|z}dz\\
&+\mathbbm{1}_{\{ 0 \leq t \leq {\nu^2}/\gamma^2 \}} e^{-\frac{\theta_0}{2} by} \int_0^{1+\mu} e^{\e_0(1+\mu-z)_+|\xi|} N_{\xi}(s,z) |\xi| e^{-\frac{1}{2}|\xi|z}dz.
\end{aligned}
\end{equation*}
By the fact that $\nu \lesssim w_s(z) $, there is
\begin{equation*}
\begin{aligned}
&\sup_{0 \leq y \leq 1+\mu} |e^{e_0(1+\mu-y)_+|\xi|} w_t(y) T_2(t-s,y) \partial_z \Delta^{-1}N_{\xi}(s,z)|_{z=0}|\\
\lesssim&  \sup_{0 \leq y \leq 1+\mu} \left(1 + \frac{1}{s^\beta}\frac{y}{\sqrt{\nu}} e^{-\frac{\theta_0}{2} by}\right) \Vert e^{\e_0(1+\mu-z)_+|\xi|} N_{\xi}(s,z)\Vert_{\mathcal{L}^1_{\mu}}\\
&+  \int_0^{1+\mu}  \left(w_s(z) + \frac{1+\mu}{s^\beta}\right) e^{\e_0(1+\mu-z)_+|\xi|} N_{\xi}(s,z) |\xi| e^{-\frac{1}{2}|\xi|z}dz\\
&+\mathbbm{1}_{\{ 0 \leq t \leq {\nu^2}/\gamma^2 \}}  \frac{1}{\sqrt{\nu}}\Vert e^{\e_0(1+\mu-z)_+|\xi|} N_{\xi}(s,z) e^{-\frac{1}{2}|\xi|z}\Vert_{\mathcal{L}^1_{\mu}}\\
&+\mathbbm{1}_{\{ 0 \leq t \leq {\nu^2}/\gamma^2 \}} \int_0^{1+\mu} e^{\e_0(1+\mu-z)_+|\xi|} N_{\xi}(s,z) |\xi| e^{-\frac{1}{2}|\xi|z}\frac{1}{\sqrt{\nu}} w_s(z)dz\\
\lesssim&\sup_{0 \leq y \leq 1+\mu}  \frac{1}{s^\beta} \Vert e^{\e_0(1+\mu-z)_+|\xi|} N_{\xi}(s,z)\Vert_{\mathcal{L}^1_{\mu}}\\
&+  \int_0^{1+\mu} e^{\e_0(1+\mu-z)_+|\xi|} N_{\xi}(s,z) w_s(z)|\xi| e^{-\frac{1}{2}|\xi|z}dz \\
&+ \frac{1}{s^\beta} \int_0^{1+\mu} e^{\e_0(1+\mu-z)_+|\xi|} N_{\xi}(s,z) dz\\
&+\mathbbm{1}_{\{ 0 \leq t \leq {\nu^2} /\gamma^2\}}  \frac{1}{\sqrt{\nu}}\Vert e^{\e_0(1+\mu-z)_+|\xi|} N_{\xi}(s,z) e^{-\frac{1}{2}|\xi|z}\Vert_{\mathcal{L}^1_{\mu}}\\
&+\mathbbm{1}_{\{ 0 \leq t \leq {\nu^2} /\gamma^2\}} \int_0^{1+\mu} e^{\e_0(1+\mu-z)_+|\xi|} N_{\xi}(s,z) |\xi| e^{-\frac{1}{2}|\xi|z}\frac{1}{\sqrt{\nu}} w_s(z)dz.\\
\end{aligned}
\end{equation*}
Finally, recalling the definition of $\Vert \cdot \Vert_{\mathcal{L}_\mu^1}$ and $\Vert \cdot \Vert_{\mathcal{L}^{\infty}_{\mu, \nu, s}}$, we have
\begin{equation*}
\begin{aligned}
&\sup_{0 \leq y \leq 1+\mu} |e^{e_0(1+\mu-y)_+|\xi|} w_t(y) T_2(t-s,y) \partial_z \Delta^{-1}N_{\xi}(s,z)|_{z=0}|\\
\lesssim&\sup_{0 \leq y \leq 1+\mu} \left( \frac{1}{s^\beta} + \mathbbm{1}_{\{ 0 \leq t \leq {\nu^2}/\gamma^2 \}}  \frac{1}{\sqrt{\nu}}\right) \Vert e^{\e_0(1+\mu-z)_+|\xi|} N_{\xi}(s,z)\Vert_{\mathcal{L}^1_{\mu}} \\
&+ \left(\mathbbm{1}_{\{ 0 \leq t \leq {\nu^2}/\gamma^2 \}}  \frac{1}{\sqrt{\nu}} + \int_0^{1+\mu} |\xi| e^{-\frac{1}{2} |\xi| z} dz \right) \Vert  e^{\e_0(1+\mu-z)_+|\xi|} N_{\xi}(s,z)\Vert_{\mathcal{L}^{\infty}_{\mu, \nu, s}}.
\end{aligned}
\end{equation*}
\textbf{(2) For the $I_2$ part:}  Applying estimate \eqref{eq3}, there is
\begin{equation*}
\begin{aligned}
&\sup_{0 \leq y \leq 1+\mu} |e^{\e_0(1+\mu -y)_+|\xi|} w_t(y) b e^{-\frac{1}{2} \theta_0 by} I_2|\\
\lesssim& \sup_{0 \leq y \leq 1+\mu} |e^{\e_0(1+\mu -y)_+|\xi|} \left(\sqrt{\nu} + \frac{y}{s^\beta} + \left(1-\frac{t\gamma^2}{\nu^2}\right)_+\right) b e^{-\frac{1}{2} \theta_0 by} \int_{1+\mu}^{\infty} e^{-|\xi| z} N_{\xi}(s,z) dz|.\\
\end{aligned}
\end{equation*}
Let $\e_0$ be small enough such that
$
e^{\e_0(1+\mu -y)_+|\xi|} e^{-\frac{1}{4}\theta_0 by} \lesssim 1
$. 
Then, we have
\begin{equation*}
\begin{aligned}
&\sup_{0 \leq y \leq 1+\mu} |e^{\e_0(1+\mu -y)_+|\xi|} w_t(y) b e^{-\frac{1}{4} \theta_0 by} I_2|\\
\lesssim& \sup_{0 \leq y \leq 1+\mu} \left| \left(\sqrt{\nu} + \frac{y}{s^\beta} + \left(1-\frac{t \gamma^2}{\nu^2}\right)_+\right) b  e^{-\frac{1}{4} \theta_0 by}\int_{1+\mu}^{\infty} e^{-|\xi| z} N_{\xi}(s,z) dz\right|\\
\lesssim&  \sup_{0 \leq y \leq 1+\mu} \left(b\sqrt{\nu}  e^{-\frac{1}{4} \theta_0 by} + \frac{1}{s^\beta} by e^{-\frac{1}{4} \theta_0 by} + \frac{1}{\sqrt{\nu}} \mathbbm{1}_{\{ 0 \leq t \leq {\nu^2}/\gamma^2\}}  b\sqrt{\nu}  e^{-\frac{1}{4} \theta_0 by}\right)\\
&\ \ \ \ \ \ \ \ \ \ \ \ \ \ \ \ \ \ \ \ \ \ \ \ \ \ \ \ \ \ \ \ \ \ \ \ \ \ \ \ \ \ \ \ \ \ \ \ \ \ \ \ \ \ \ \ \ \ \ \ \ \ \ \ \ \ \ \ \ \ \ \ \ \ \ \ \ \ \ \ \ \ \ \ \ \ \ \ \ \  \ \ \ \times\Vert z N_{\xi}(s,z) \Vert_{L^2(z \geq 1+\mu)}\\
\lesssim&\sup_{0 \leq y \leq 1+\mu} \left(\frac{1}{\sqrt{\nu}} \mathbbm{1}_{\{ 0 \leq t \leq {\nu^2}/\gamma^2\}} + \frac{1}{s^{\beta}}\right)\Vert z N_{\xi}(s,z) \Vert_{L^2(z \geq 1+\mu)}.
\end{aligned}
\end{equation*}
The rest part of the estimate involving $T_2$ is the same as that of $T_1$. Hence we have proved the lemma for the case $(i,j)=(0,0)$.  The proof for the case $(i,j)=(0,1)$ follows in the same manner since $y\partial_y G$ and $G$ share similar bounds. For $(i,j)=(1,0)$, we apply the same estimate as in the case $(i,j)=(0,0)$ to $\partial_x N(s,z)$.
\end{proof}

\begin{Lemma}\label{X4}
Let $\mu \in (0, \mu_0- \gamma s)$ be arbitrary. There is
\begin{equation}
\begin{aligned}
&\sum_{0 \leq i+j \leq 2} \left\Vert \partial_x^i (y \partial_y)^j \int_0^{\infty} G(t,y,z) \omega_0(z) \right\Vert_{X_{\mu,t}}\\
\lesssim& \sum_{0 \leq i+j \leq 2} \Vert \partial_x^i (y\partial_y)^j \omega_0\Vert_{X_{\mu,t}} + \sum_{0 \leq i+j \leq 2} \sum_{\xi} \Vert \xi^i \partial_y^i \omega_{0,\xi} \Vert_{L^{\infty}_{y \geq 1+\mu}}.
\end{aligned}
\end{equation}
\end{Lemma}
\begin{proof}
Recalling that
\begin{equation*}
G_{\xi}(t,y,z) = H_{\xi}(t,y,z) + H_{\xi}(t,y,-z) + R_{\xi}(t,y,z),
\end{equation*}
we have
\begin{equation*}
\begin{aligned}
&\left\Vert (\partial_x)^i (y \partial_y)^j \int_{\mathbb{T} \times \mathbb{R}} G(t,y,z) \omega_0(z) dz \right\Vert_{X_{\mu,t}}\\
=& \sum_{\xi} \left\Vert e^{\e_0 (1+\mu-y)_+|\xi|}  (y\partial_y)^j \int_{\mathbb{R}^+} H_{\xi}(t,y,z) \xi^i \omega_{0,\xi}(z) dz\right\Vert_{\mathcal{L}^{\infty}_{\mu, \nu ,t}}\\
&+  \sum_{\xi} \left\Vert e^{\e_0 (1+\mu-y)_+|\xi|}  (y\partial_y)^j \int_{\mathbb{R}^+} \tilde{H}_{\xi}(t,y,z) \xi^i \omega_{0,\xi}(z) dz\right\Vert_{\mathcal{L}^{\infty}_{\mu, \nu ,t}}\\
&+ \sum_{\xi} \left\Vert e^{\e_0 (1+\mu-y)_+|\xi|}  (y\partial_y)^j \int_{\mathbb{R}^+} R_{\xi}(t,y,z) \xi^i \omega_{0,\xi}(z) dz\right\Vert_{\mathcal{L}^{\infty}_{\mu, \nu ,t}}\\
=&J_1+ J_1' + J_2.
\end{aligned}
\end{equation*}
The estimate of \[\left\Vert e^{\e_0 (1+\mu-y)_+|\xi|}  (y\partial_y)^j \int_{\mathbb{R}^+} H_{\xi}(t,y,z) \xi^i \omega_{0,\xi}(z) dz\right\Vert_{\mathcal{L}^{\infty}_{\mu, \nu ,t}}\] follows exactly the same as Lemma \ref{l1}, where we simply notice that $s=t$ is a special case of $s\leq t$. We next proceed to estimate the quantity
 \[\left\Vert e^{\e_0 (1+\mu-y)_+|\xi|}  (y\partial_y)^j \int_{\mathbb{R}^+} R_{\xi}(t,y,z) \xi^i \omega_{0,\xi}(z) dz\right\Vert_{\mathcal{L}^{\infty}_{\mu, \nu ,t}}.\] 
By the same argument in Lemma \ref{l2}, we only need to consider the case when $(i,j)=(0,1)$. Similar to Lemma \ref{l2}, there is
\begin{equation*}
\begin{aligned}
&\int_0^{\frac{3}{4}y} b e^{-\frac{\theta}{2}b(y+z)} \frac{w_t(y)}{w_t(z)} dz\\
=&\int_0^{\frac{3}{4}y} b e^{-\frac{\theta}{2}b(y+z)} \frac{w\left(y \frac{\nu^{2\beta}}{\gamma^{2\beta} t^\beta}\right) + \left(1-\frac{t\gamma^2}{\nu^2}\right)_+}{w\left(z \frac{\nu^{2\beta}}{\gamma^{2\beta} t^\beta}\right) + \left(1- \frac{t\gamma^2}{\nu^2}\right)_+}dz\\
\lesssim& \int_0^{t^{\beta} \nu^{-2\beta}} b e^{-\frac{\theta}{2}b(y+z)} \frac{y\frac{\nu^{2\beta}}{\gamma^{2\beta} t^\beta}+ \sqrt{\nu} + \left(1-\frac{t \gamma^2}{\nu^2}\right)_+}{\sqrt{\nu} + \left(1- \frac{t \gamma^2}{\nu^2}\right)_+}dz + \int_{t^{\beta} \nu^{-2\beta}}^{\frac{3}{4}y} b e^{-\frac{\theta_0}{2}b(y+z)} dz\\
\lesssim& \int_0^{t^{\beta} \nu^{-2\beta}} b e^{-\frac{\theta}{2}b(y+z)} \frac{y\frac{\nu^{2\beta}}{\gamma^{2\beta} t^{\beta}}+ \sqrt{\nu} + \left(1-\frac{t \gamma^2}{\nu^2}\right)_+}{\sqrt{\nu} + \left(1- \frac{t \gamma^2}{\nu^2}\right)_+}dz + \int_{t^{\beta} \nu^{-2\beta}}^{\frac{3}{4}y} b e^{-\frac{\theta_0}{2}b(y+z)} dz\\
:=& J_{21} + J_{22}.
\end{aligned} 
\end{equation*}
We first estimate $J_{21}$ by splitting it into two cases.
\begin{enumerate}
\item When $t < \nu^2/2\gamma^2$, we have
\begin{equation*}
\begin{aligned}
J_{21} \lesssim& \int_0^{t^{\beta} \nu^{-2\beta}}b e^{-\frac{\theta}{2}b(y+z)} \left(1+ y \frac{\nu^{2\beta}}{t^{\beta}}\right) dz\\
=&\int_0^{t^{\beta} \nu^{-2\beta}}b e^{-\frac{\theta}{2}bz}  dz \cdot e^{-\frac{\theta_0}{2} by} + \int_0^{t^{\beta} \nu^{-2\beta}} e^{-\frac{\theta}{2}bz}   \frac{\nu^{2\beta}}{t^{\beta}} dz \cdot by e^{-\frac{\theta_0}{2}by}\\
\lesssim& \int_0^{\infty}b e^{-\frac{\theta}{2}bz}  dz + t^{\beta} \nu^{-2\beta} \frac{\nu^{2\beta}}{t^{\beta}} \lesssim1.
\end{aligned}
\end{equation*}
\item When $t> \nu^2/2\gamma^2$, the fact that
$
\frac{1}{t^\beta} \lesssim \nu^{-2\beta}
$ 
implies
$
y\frac{\nu^{2\beta}}{t^\beta} \lesssim y.
$
Therefore, we obtain
\begin{equation*}
\begin{aligned}
J_{21} \lesssim& \int_0^{t^{\beta} \nu^{-2\beta}} be^{-\frac{\theta_0}{2}b(y+z)}\frac{y}{\sqrt{\nu}} dz
\lesssim by e^{-\frac{\theta_0}{2}by} \int_0^{t^{\beta} \nu^{-2\beta}} e^{-\frac{\theta_0}{2}bz} dz \cdot \frac{1}{\sqrt{\nu}}\\
\lesssim& \int_0^{\infty} e^{-\frac{\theta_0}{2}bz} d\frac{\theta_0}{2}bz \cdot \frac{2}{\theta_0 b\sqrt{\nu}} \lesssim 1.
\end{aligned}
\end{equation*}
\end{enumerate}
To estimate $J_{22}$, we simply notice that
\begin{equation*}
J_{22}  \lesssim \int_0^{\infty} be^{-\frac{\theta_0}{2}bz} dz e^{-\frac{\theta_0}{2} by} \lesssim 1,
\end{equation*}
and the proof is concluded.
\end{proof}

\section{Estimate of the analytic norm and Sobolev norm}
In this section, the estimates hold for all $t \in[0, T]$. 
\begin{Lemma}
	\label{X:anal}
Let
\begin{equation}
\mu_1=\mu +\frac{1}{4}(\mu_0- \mu-\gamma \sqrt{s}),\ \ \  \mu_2= \mu + \frac{1}{2}(\mu_0- \mu-\gamma \sqrt{s}).
\end{equation}
We have
\begin{equation}
\begin{aligned}
&(\mu_0-\mu-\gamma \sqrt{s}) \sum_{i+j=2} \left\Vert \partial_x^i (y\partial_y)^j  \int_0^{\infty} G(t-s,y,z) N(s,z) dz\right\Vert_{X_{\mu,t}}\\
& + \sum_{i+j \leq 1} \left\Vert \partial_x^i (y\partial_y)^j  \int_0^{\infty} G(t-s,y,z) N(s,z) dz \right\Vert_{X_{\mu_1,t}}\\
\lesssim&\sum_{i+j \leq 1} \Vert \partial_x^i (y\partial_y)^j N(s) \Vert_{X_{\mu_2,t}} + \frac{1}{\sqrt{\mu_0- \mu-\gamma \sqrt{s}}} \sum_{i+j \leq 1} \Vert \partial_x^i \partial_y^j N(s) \Vert_{S_{\mu_2}},
\end{aligned}
\end{equation}

\begin{equation}
\begin{aligned}
&(\mu_0-\mu-\gamma \sqrt{s})  \sum_{i+j=2} \Vert \partial_x^i (y\partial_y)^j G(t-s,y,0) B(s) \Vert_{X_{\mu,t}}\\
& + \sum_{i+j \leq2}  \Vert \partial_x^i (y\partial_y)^j G(t-s,y,0) B(s) \Vert_{X_{\mu,t}}\\
\lesssim &\frac{1}{\sqrt{t-s}} \left(\frac{1}{\sqrt{\nu}} \mathbbm{1}_{\{0 \leq t \leq {\nu^2}/\gamma^2 \}} + \frac{1}{s^\beta}\right) (\sum_{i\leq1}\Vert \partial_x^i N(s) \Vert_{Y_{\mu_1}} + \Vert \partial_x^i N(s) \Vert_{S_{\mu_1}})\\
&+ \left(\frac{1}{\sqrt{\nu}} \mathbbm{1}_{\{0 \leq t \leq {\nu^2}/\gamma^2 \}} +1\right) \Vert \partial_x^i N(s) \Vert_{X_{\mu_1,s}},
\end{aligned}
\end{equation}
and
\begin{equation}
\begin{aligned}
\sum_{i+j \leq 2} \left\Vert \partial_x^i (y\partial_y)^j \int_0^{\infty} G(t,y,z) \omega_0(z)dz\right\Vert_{X_{\mu,t}} \lesssim& \sum_{i+j \leq 2}\Vert \partial_x^i (y\partial_y)^j \omega_0 \Vert_{X_{\mu,t}}\\
& + \sum_{i+j \leq 2} \sum_{\xi} \Vert \xi^i \partial_y^j \omega_{0, \xi} \Vert_{L^{\infty}(y \geq 1+\mu)}.
\end{aligned}
\end{equation}

\end{Lemma}
\begin{proof}
Our first observation is that by checking the proof of Lemma A.3 in \cite{KukVicWan20}, the analyticity recovery for the $X$ norm still holds for our weight function $w_t(y)$, that is
\begin{equation*}
\begin{aligned}
&\sum_{i+j=2} \left\Vert \partial_x^i (y\partial_y)^j  \int_0^{\infty} G(t-s,y,z) N(s,z) dz\right\Vert_{X_{\mu,t}}\\
\lesssim& \frac{1}{\mu_0-\mu -\gamma \sqrt{s}}  \sum_{i+j=1} \left\Vert \partial_x^i (y\partial_y)^j  \int_0^{\infty} G(t-s,y,z) N(s,z) dz\right\Vert_{X_{\mu_1,t}}.
\end{aligned}
\end{equation*} 
In the previous section, we have shown that the lemma holds when $t < \nu^2/\gamma^2$. When $t\ge\nu^2/\gamma^2$, there is 
\begin{equation*}
w_t(y)=w(y).
\end{equation*}
If $s\ge \nu^2/\gamma^2$, the weight function is exactly the same as \cite{KukVicWan20} and so we can apply the result there directly. While for $s< \nu^2/\gamma^2$, notice that 
\begin{equation*}
w_s(z) \geq w(z),
\end{equation*}
for any $s>0$.  Therefore, we have
\begin{equation*}
\frac{w_t(y)}{w_s(z)} \leq \frac{w(y)}{w(z)},
\end{equation*}
and we again apply the result in \cite{KukVicWan20}.
\end{proof}

Since our definition of $Y$ norm is the same as \cite{KukVicWan20}, we borrow the following lemma directly. 
\begin{Lemma}
For any $\mu \in (0, \mu_0-\gamma \sqrt{s})$, we have 

\begin{equation}
\begin{aligned}
&(\mu_0-\mu-\gamma \sqrt{s})\sum_{i+j =2} \left\Vert \partial_x^i (y \partial_y)^j  \int_0^{\infty} G(t-s,y,z) N(s,z)\,dz\right\Vert_{Y_{\mu}} \\
&+ \sum_{i+j \leq 1} \left\Vert  \partial_x^i(y \partial_y)^j \int_0^{\infty} G(t-s,y,z) N(s,z)\, dz \right\Vert_{Y_{\mu_1}}\\
\lesssim& \sum_{i+j \leq 1} \Vert \partial_x^i (y \partial_y)^j N(s) \Vert_{Y_{\mu_1}} + \sum_{i+j \leq 1} \Vert \partial_x^i (y \partial_y)^j N(s)\Vert_{S_{\mu_1}},
\end{aligned}
\end{equation}

\begin{equation}
\begin{aligned}
&(\mu_0-\mu-\gamma \sqrt{s})\sum_{i+j =2} \Vert \partial_x^i (y \partial_y)^j    G(t-s,y,0) B(s)\Vert_{Y_{\mu}} \\
&+ \sum_{i+j \leq 1} \Vert  \partial_x^i(y \partial_y)^j  G(t-s,y,0) B(s)  \Vert_{Y_{\mu_1}}\\
\lesssim& \sum_{i+j \leq 1} \Vert \partial_x^i (y \partial_y)^j N(s) \Vert_{Y_{\mu_1}} + \sum_{i+j \leq 1} \Vert \partial_x^i (y \partial_y)^j N(s)\Vert_{S_{\mu}},
\end{aligned}
\end{equation}
and
\begin{equation}
\begin{aligned}
&\sum_{i+j \leq 2} \Vert \partial_x^i (y \partial_y)^j    G(t-s,y,z) \omega_0(z)\Vert_{Y_{\mu}} \\
\lesssim& \sum_{i+j \leq 2} \Vert \partial_x^i (y \partial_y)^j \omega_0 \Vert_{Y_{\mu}} + \sum_{i+j \leq 2} \sum_{\xi} \Vert \xi^i  \partial_y^j \omega_{0,\xi}\Vert_{L^1(y \geq 1+\mu)}.
\end{aligned}
\end{equation}

\end{Lemma}
 
At last, we estimate the nonlinear term $N(s)$.
\begin{Lemma}
For any $\mu \in (0, \mu_0 - \gamma s)$, we have

\begin{equation}\label{eq7}
\begin{aligned}
&\sum_{i+j \leq 1} \Vert \partial_x^i (y \partial_y)^j N(s) \Vert_{X_{\mu,s}}\\
 \lesssim& \sum_{i \leq 1} (\Vert \partial_x^i \omega(s) \Vert_{Y_\mu} + \Vert \partial_x^i \omega(s) \Vert_{S_\mu}) \sum_{i+j \leq 2} \Vert \partial_x^i (y \partial_y)^j \omega(s)\Vert_{X_{\mu,s}}\\
&+\sum_{i \leq 2}  (\Vert \partial_x^i \omega(s) \Vert_{Y_\mu} + \Vert \partial_x^i \omega(s) \Vert_{S_\mu}) \sum_{i+j \leq 1} \Vert \partial_x^i (y \partial_y)^j \omega(s)\Vert_{X_{\mu,s}}\\
&+ \Vert \omega(s) \Vert_{X_{\mu,s}} \sum_{i+j=1} \Vert \partial_x^i (y\partial_y)^j \omega(s)  \Vert_{X_{\mu,s}},
\end{aligned}
\end{equation}
\begin{equation}\label{eq5}
\begin{aligned}
&\sum_{i+j \leq 1} \Vert \partial_x^i (y \partial_y)^j N(s) \Vert_{Y_{\mu}}\\
 \lesssim& \sum_{i \leq 1} (\Vert \partial_x^i \omega(s) \Vert_{Y_\mu} + \Vert \partial_x^i \omega(s) \Vert_{S_\mu}) \sum_{i+j \leq 2} \Vert \partial_x^i (y \partial_y)^j \omega(s)\Vert_{Y_{\mu}}\\
&+\sum_{i \leq 2}  (\Vert \partial_x^i \omega(s) \Vert_{Y_\mu} + \Vert \partial_x^i \omega(s) \Vert_{S_\mu}) \sum_{i+j \leq 1} \Vert \partial_x^i (y \partial_y)^j \omega(s)\Vert_{Y_{\mu}}\\
&+ \Vert \omega(s) \Vert_{X_{\mu,s}} \sum_{i+j=1} \Vert \partial_x^i (y\partial_y)^j \omega(s)  \Vert_{Y_{\mu}},
\end{aligned}
\end{equation}
and
\begin{equation}\label{eq6}
\sum_{i+j \leq 1} \Vert \partial_x^i \partial_y^j N(s) \Vert_{S_{\mu}} \lesssim \normmm{\omega(s)}_{s} \sum_{i+j \leq 3} \Vert  \partial_x^i \partial_y^j \omega(s)\Vert_S.
\end{equation}

\end{Lemma}

\begin{proof}
The proof of estimate \eqref{eq5} and \eqref{eq6} follows from Lemma 6.4 and Lemma 6.5 in \cite{KukVicWan20}. We only give a proof of  \eqref{eq7} here. First consider the case  $i+j=1$. Recalling
\begin{equation*}
N(s)= u_1 \partial_x \omega + u_2  \partial_y \omega,
\end{equation*} 
and applying Leibniz rule, we have 
\begin{equation*}
\begin{aligned}
&\left\Vert \partial_x^i (y \partial_y)^j N(s)\right\Vert_{X_{\mu,s}}\\
=& \sum_{\xi}\left\Vert e^{\e_0(1+ \mu -y)_+|\xi|} (-\partial_x^i (y\partial_y)^j N(s))_{\xi}\right\Vert_{\mathcal{L}^{\infty}_{\mu, \nu,s}}\\
\lesssim&  \sum_{\xi}\left\Vert e^{\e_0(1+ \mu -y)_+|\xi|} (-\partial_x^i (y\partial_y)^j u_1  \partial_x \omega)_{\xi}\right\Vert_{\mathcal{L}^{\infty}_{\mu, \nu,s}} \\
&\ \ \ \ \ \ \ \ \ \ \ \ \ \ \ \ \ \ \ \ \ \ \ \ \ \ \  \ \ \ \ \ \ \ +  \sum_{\xi}\left\Vert e^{\e_0(1+ \mu -y)_+|\xi|} (u_1(y \partial_y)^j \partial_x^{i+1} \omega)_{\xi}\right\Vert_{\mathcal{L}^{\infty}_{\mu, \nu,s}}\\
&+  \sum_{\xi}\left\Vert e^{\e_0(1+ \mu -y)_+|\xi|} \left((y\partial_y)^j \frac{\partial_x^i u_2}{y}  y \partial_y \omega\right)_{\xi}\right\Vert_{\mathcal{L}^{\infty}_{\mu, \nu,s}}\\
&\ \ \ \ \ \ \ \ \ \ \ \ \ \ \ \ \ \ \ \ \ \ \ \ \ \ \ \ \ \ \ \ \ \ +  \sum_{\xi}\left\Vert e^{\e_0(1+ \mu -y)_+|\xi|} \left(\frac{u_2}{y}\partial_x^i (y\partial_y)^j  \omega\right)_{\xi}\right\Vert_{\mathcal{L}^{\infty}_{\mu, \nu,s}}\\
\lesssim& \sum_{\xi}\left\Vert e^{\e_0(1+ \mu -y)_+|\xi|} \sum_{\eta} (-\partial_x^i (y\partial_y)^j u_1)_{\xi-\eta}  (\partial_x \omega)_{\eta}\right\Vert_{\mathcal{L}^{\infty}_{\mu, \nu,s}}\\
&\ \ \ \ \ \ \ \ \ \ \ \ \ \ \ \ \ \ \ \ \ \ \ \ \ \ +\sum_{\xi}\left\Vert e^{\e_0(1+ \mu -y)_+|\xi|} \sum_{\eta}((y \partial_y)^j \partial_x^{i+1} \omega)_{\eta}  (u_1)_{\xi-\eta}\right\Vert_{\mathcal{L}^{\infty}_{\mu, \nu,s}} \\
&+ \sum_{\xi}\left\Vert e^{\e_0(1+ \mu -y)_+|\xi|} \sum_{\eta}\left((y\partial_y)^j \frac{\partial_x^i u_2}{y}\right)_{\xi-\eta}  (y \partial_y \omega)_{\eta}\right\Vert_{\mathcal{L}^{\infty}_{\mu, \nu,s}}\\
&\ \ \ \ \ \ \ \ \ \ \ \ \ \ \ \ \ \ \ \ \ \ \ \ \ \ +\sum_{\xi}\left\Vert e^{\e_0(1+ \mu -y)_+|\xi|} \sum_{\eta}\left(\frac{u_2}{y}\right)_{\xi-\eta}  (\partial_x^i (y\partial_y)^j  \omega)_{\eta}\right\Vert_{\mathcal{L}^{\infty}_{\mu, \nu,s}}.
\end{aligned}
\end{equation*}
Noticing that 
\begin{equation*}
\begin{aligned}
e^{\e_0(1+\mu-y)_+|\xi|} \leq& e^{\e_0(1+\mu-y)_+(|\eta|+ |\xi-\eta|)}\\
=&e^{\e_0(1+\mu-y)_+|\eta|} \cdot e^{\e_0(1+\mu-y)_+|\xi-\eta|},
\end{aligned}
\end{equation*}
we have
\begin{equation*}
\begin{aligned}
&\sum_{\xi}\left\Vert e^{\e_0(1+ \mu -y)_+|\xi|} (-\partial_x^i (y\partial_y)^j u_1)_{\xi-\eta}  (\partial_x \omega)_{\eta}\right\Vert_{\mathcal{L}^{\infty}_{\mu, \nu,s}}\\
\lesssim&\sum_{\xi}\left\Vert \sum_{\eta} e^{\e_0(1+ \mu -y)_+|\xi-\eta|} (-\partial_x^i (y\partial_y)^j u_1)_{\xi-\eta}  e^{\e_0(1+ \mu -y)_+|\eta|}(\partial_x \omega)_{\eta}\right\Vert_{\mathcal{L}^{\infty}_{\mu, \nu,s}}\\
\lesssim&\sum_{\xi}\sup_{y \in \Omega_{\mu}}\left| \sum_{\eta} e^{\e_0(1+ \mu -y)_+|\xi-\eta|} (-\partial_x^i (y\partial_y)^j u_1)_{\xi-\eta}  e^{\e_0(1+ \mu -y)_+|\eta|}(\partial_x \omega)_{\eta}\right|\\
\lesssim&\sum_{\xi} \sum_{\eta}\sup_{y \in \Omega_{\mu}}\left| e^{\e_0(1+ \mu -y)_+|\xi-\eta|} (-\partial_x^i (y\partial_y)^j u_1)_{\xi-\eta}|  \sup_{y \in \Omega_{\mu}} |e^{\e_0(1+ \mu -y)_+|\eta|}(\partial_x \omega)_{\eta}\right|\\
\lesssim&\sum_{\eta} \sum_{\xi}\sup_{y \in \Omega_{\mu}}\left| e^{\e_0(1+ \mu -y)_+|\xi-\eta|} (-\partial_x^i (y\partial_y)^j u_1)_{\xi-\eta}|  \sup_{y \in \Omega_{\mu}} |e^{\e_0(1+ \mu -y)_+|\eta|}(\partial_x \omega)_{\eta}\right|\\
\lesssim& \Vert \partial_x\omega(s)\Vert_{X_{\mu,s}} \sum_{\eta}\sup_{y \in \Omega_{\mu}}\left| e^{\e_0(1+ \mu -y)_+|\eta|} (-\partial_x^i (y\partial_y)^j u_1)_{\eta}\right|.
\end{aligned}
\end{equation*}
Applying similar arguments to the other terms, we arrive at the following estimate,
\begin{equation*}
\begin{aligned}
&\Vert \partial_x^i (y \partial_y)^j N(s)\Vert_{X_{\mu,s}}\\
 \lesssim& \Vert \partial_x\omega(s)\Vert_{X_{\mu,s}} \sum_{\xi}\sup_{y \in \Omega_{\mu}}| e^{\e_0(1+ \mu -y)_+|\eta|} (-\partial_x^i (y\partial_y)^j u_1)_{\xi}|\\
&+\Vert y \partial_y \omega(s) \Vert_{X_{\mu,s}} \sum_{\xi} \sup_{y \in \Omega_\mu}\left| e^{\e_0(1+\mu-y)_+|\xi|} (y \partial_y)^j \left(\frac{\partial_x^i u_2}{y}\right)_{\xi}\right|\\
&+\Vert \partial_x^{i+1} (y \partial_y)^j w(s) \Vert_{X_{\mu,s}} \sum_{\xi} \sup_{y \in \Omega_{\mu}}|e^{\e_0(1+ \mu -y)_+|\xi|} (u_1)_{\xi}|\\
&+\Vert \partial_x^{i} (y \partial_y)^{j+1} w(s) \Vert_{X_{\mu,s}} \sum_{\xi} \sup_{y \in \Omega_{\mu}}\left|e^{\e_0(1+ \mu -y)_+|\xi|} \left(\frac{u_2}{y}\right)_{\xi}\right|.\\
\end{aligned}
\end{equation*}
Since $\Vert \cdot \Vert_{X_{\mu}} \leq \Vert \cdot \Vert_{X_{\mu,s}}$, by the proof of Lemma 6.3 in \cite{KukVicWan20}, we have, for $i + j \leq 1$,
\begin{equation*}
\sum_{\xi}\sup_{y \in \Omega_{\mu}}| e^{\e_0(1+ \mu -y)_+|\eta|} (-\partial_x^i (y\partial_y)^j u_1)_{\xi}| \lesssim \Vert  \partial_x^{i+j}\Vert_{Y_\mu} + \Vert \partial_x^{i+j} \omega(s)\Vert_{S_{\mu}} + j \Vert \omega(s) \Vert_{X_{\mu,s}},
\end{equation*}
and
\begin{equation*}
\sum_{\xi} \sup_{y \in \Omega} e^{\e_0(1+\mu-y)_+|\xi|} \left|(y \partial_y)^j \left(\frac{\partial_x^i u_2}{y}\right)_\xi\right| \lesssim \Vert \partial_x^{i+1} \omega\Vert_{Y_{\mu}} + \Vert \partial_x^{i+1} \omega \Vert_{S_{\mu}}.
\end{equation*}
The case when $i=j=0$ follows exactly the same as the case when $i+j=1$.
\end{proof}

\section{Proof of the main Theorems}\label{Proof}
Since $w_t(y)\gtrsim1$ for $y \in[1/4,1/2]$, $\gamma \in (0,1)$, and all $t>0$, we apply the proof in section 7 of \cite{KukVicWan20} to get the following lemma.
\begin{Lemma}\label{Sobolev}
For any $t \in (0, \frac{\mu_0}{2\gamma})$, $\phi(y)=y \bar{\phi}(y)$, where $\bar{\phi} \in C^{\infty}$, nondecreasing function such that $\bar{\phi}=0$ when $y \in [0, 1/4]$, $\bar{\phi}=1$ when $y \in [1/2,\infty)$, there is 
\begin{equation}
\sum_{i+j \leq 3} \Vert y \partial_x^i \partial_y^j \omega(t) \Vert_{L^2(y \geq 1/2)} \lesssim \left(t + t \sup \normmm{\omega(s)}_{s}^2 + \sum_{i+j \leq 3} \Vert \phi \partial_x^i \partial_y^j \omega_0 \Vert^2_{L^2(H)}\right) e^{Ct(1 + \sup_{s \in [0,t]} \normmm{\omega(s)}_s)},
\end{equation}
where $C>0$ is independent of $\gamma$.
\end{Lemma}

\subsection{Proof of Theorem \ref{thm0}}
We next prove Theorem \ref{thm0}.
\begin{proof}
Define 
\begin{equation*}
\tilde{M}= \sum_{i+j \leq 2} \Vert \partial_x^i(y\partial_y)^j \omega_0 \Vert_{X_{\mu_0,0}} + \sum_{i+j\leq 2} \sum_{\xi} \Vert \partial_x^i \partial_y^j \omega_{0,\xi}\Vert_{L^{\infty}_{(y \geq 1+\mu_0)}},
\end{equation*}
and
\begin{equation*}
\bar{M}= \sum_{i+j \leq 2} \Vert \partial_x^i(y\partial_y)^j \omega_0 \Vert_{Y_{\mu_0}} + \sum_{i+j\leq 2} \sum_{\xi} \Vert \partial_x^i \partial_y^j \omega_{0,\xi}\Vert_{L^{1}_{(y \geq 1+\mu_0)}}.
\end{equation*}
Then by our condition on the initial value, the estimate
\begin{equation*}
\sum_{\xi} \Vert \omega_{\xi} \Vert_{L^{\infty}_{(y\geq 1+\mu)}} \leq \sum_{\xi} \Vert y \partial_y \omega_{\xi} \Vert_{L^2(y \geq 1+\mu)} = \Vert \partial_y \omega \Vert_{S_\mu},
\end{equation*}
and Lemma A.1 in \cite{KukVicWan20}, we have
\begin{equation*}
\tilde{M} + \bar{M} \lesssim M.
\end{equation*}
Let $t \leq \mu_0/2\gamma$, $s \in (0,t)$, and $\mu < \mu_0-\gamma t$. Note $\Vert \cdot \Vert_{X_{\mu,t}} \leq \Vert \cdot \Vert_{X_{\mu_0,t}}$. We first derive estimates for the $X$ analytic norm.
For the case $i+j =2$, there is 
\begin{equation*}
\begin{aligned}
&\sum_{i+j=2} \Vert \partial_x^i (y\partial_y)^j \omega(t)\Vert_{X_{\mu,t}}\\
=&\sum_{i + j =2} \left\Vert \partial_x^i(y\partial_y)^j  \int_0^{\infty}G(t,y,z) \omega_0(z) dz \right\Vert_{X_{\mu,t}} + \sum_{i+j=2} \left\Vert \partial_x^i (y\partial_y)^j \int_0^t G(t-s,y,0) B(s) ds\right\Vert_{X_{\mu,t}}\\
&+\sum_{i + j =2} \left\Vert \partial_x^i(y\partial_y)^j  \int_0^t\int_0^{\infty}G(t,y,z) N(s,z) dz ds \right\Vert_{X_{\mu,t}}\\
\lesssim&\sum_{i + j =2} \left\Vert \partial_x^i(y\partial_y)^j  \int_0^{\infty}G(t,y,z) \omega_0(z) dz  \right\Vert_{X_{\mu_0,t}} +\sum_{i+j=2} \left\Vert \partial_x^i (y\partial_y)^j \int_0^t G(t-s,y,0) B(s) ds\right\Vert_{X_{\mu,t}}\\
&+\sum_{i + j =2}  \int_0^t\left\Vert \partial_x^i(y\partial_y)^j \int_0^{\infty}G(t,y,z) N(s,z) dz \right\Vert_{X_{\mu,t}} ds. \\
\end{aligned}
\end{equation*}
Noticing that $\Vert \cdot \Vert_{X_{\mu_0,t}} \lesssim \Vert \cdot \Vert_{\mu_0,0}$ and applying Lemma~\ref{X:anal}, we have
\begin{equation*}
\begin{aligned}
&\sum_{i+j=2} \Vert \partial_x^i (y\partial_y)^j \omega(t)\Vert_{X_{\mu,t}}\\
\lesssim& \sum_{i+j \leq 2} \Vert \partial_x^i (y\partial_y)^j \omega_0 \Vert_{X_{\mu_0,0}} +  \sum_{i+j \leq 2} \sum_{\xi} \Vert \xi^i \partial_y^j \omega_{0,\xi} \Vert_{L^{\infty}(y \geq 1+\mu)}\\
&+\int_0^{t} \frac{1}{\mu_0-\mu -\gamma \sqrt{s}}\Big(\frac{1}{\sqrt{t-s}} \left(\frac{1}{\sqrt{\nu}} \mathbbm{1}_{\{0 \leq t \leq {\nu^2}/\gamma^2\}} + \frac{1}{s^\beta}\right) (\sum_{i\leq1}\Vert \partial_x^i N(s) \Vert_{Y_{\mu_1}} + \Vert \partial_x^i N(s) \Vert_{S_{\mu_1}})\\
&\ \ \ \ \ \ \ \ \ \ \ \ \ \ \ \ \ \ \ \ \ \ \ \ \ \ \ \ \ + \left(\frac{1}{\sqrt{\nu}} \mathbbm{1}_{\{0 \leq t \leq {\nu^2}/\gamma^2 \}} +1\right) \Vert \partial_x^i N(s)  \Vert_{X_{\mu_1,s}}\Big) ds\\
& + \int_0^t \frac{1}{\mu_0-\mu -\gamma \sqrt{s}}\sum_{i+j \leq 1} \Vert \partial_x^i (y\partial_y)^j N(s) \Vert_{X_{\mu_2,t}} \\
&\ \ \ \ \ \ \ \ \ \ \ \ \ \ \ \ \ \ \ \ \ \ \ \ \ \ \ \ \ \ + \frac{1}{(\mu_0- \mu-\gamma \sqrt{s})^{3/2}} \sum_{i+j \leq 1} \Vert \partial_x^i \partial_y^j N(s) \Vert_{S_{\mu_2}}ds.\\
\end{aligned}
\end{equation*}
Recall the definition of the norm $\normmm{\cdot}_s$, there is
\begin{equation*}
\begin{aligned}
&\sum_{i+j=2} \Vert \partial_x^i (y\partial_y)^j \omega(t)\Vert_{X_{\mu,t}}\\
\lesssim& \tilde{M} + \Big(\int_0^t \frac{1}{(\mu_0- \mu -\gamma \sqrt{s})^{1+\a}} \frac{1}{\sqrt{t-s}} \left(\frac{1}{\sqrt{\nu}} \mathbbm{1}_{\{0 \leq t \leq {\nu^2}/\gamma^2\}} + \frac{1}{s^\beta}\right) ds \\
&\ \ \ \ \ \ \ \ \ \ \ \ \ \ \ \ \ \ \ \ \ \ \ \ \ \ \ \ \ \  \ 
+\int_0^t \frac{1}{(\mu_0 - \mu - \gamma \sqrt{s})^{3/2 + \a}} \left(\frac{1}{\sqrt{\nu}} \mathbbm{1}_{\{0 \leq t \leq {\nu^2}/\gamma^2\}} + 1\right) ds\Big) \cdot 
 \sup_{s \in [0,t]}\normmm{\omega(s)}^2_s\\
:=& \tilde{M} + (I_1 + I_2 + I_3 +I_4) \cdot \sup_{s \in [0,t]}\normmm{\omega(s)}^2_s.\\
\end{aligned}
\end{equation*}
We have the following estimates for $I_1$, $I_2$, $I_3$ and $I_4$ whose proof will be given in the Appendix, 
\begin{equation*}
I_1 \lesssim \frac{1}{\sqrt{\gamma}} \frac{1}{(\mu_0-\mu-\gamma \sqrt{t})^{1/2 + \a}},
\end{equation*}
\begin{equation*}
I_2 \lesssim \frac{1}{\sqrt{\gamma}} \frac{1}{\mu_0-\mu - \gamma \sqrt{t}} \cdot t^{1/4-\beta},
\end{equation*}
\begin{equation*}
I_3 \lesssim \frac{1}{\sqrt{\gamma}} \frac{1}{(\mu_0-\mu-\gamma \sqrt{t})^{1/2 + \a}},
\end{equation*}
and
\begin{equation*}
I_4 \lesssim \frac{1}{\sqrt{\gamma}} \frac{1}{(\mu_0 - \mu -\gamma \sqrt{t})^{1/2+\a}},
\end{equation*}
for $\beta <1/4$ with $\alpha>0$. Therefore, we have
\begin{equation}\label{eqX1}
\sum_{i+j=2} \Vert \partial_x^i (y\partial_y)^j \omega(t)\Vert_{X_{\mu,t}} \lesssim \tilde{M} +  \frac{1}{\sqrt{\gamma}} \frac{1}{(\mu_0 - \mu - \gamma \sqrt{t})^{1/2 + \a}} \cdot  \sup_{s \in [0,t]}\normmm{\omega(s)}^2_s.
\end{equation}
 For the case $i+j \leq 1$, we similarly have
\begin{equation*}
\begin{aligned}
&\sum_{i+j\leq 1} \Vert \partial_x^i (y\partial_y)^j \omega(t)\Vert_{X_{\mu,t}}\\
\lesssim& \tilde{M} + (\int_0^t \frac{1}{(\mu_0- \mu -\gamma \sqrt{s})^{\a}} \frac{1}{\sqrt{t-s}} (\frac{1}{\sqrt{\nu}} \mathbbm{1}_{\{0 \leq t \leq {\nu^2}/\gamma^2\}} + \frac{1}{s^\beta}) ds \\
&\ \ \ \ \ \ 
+\int_0^t \frac{1}{(\mu_0 - \mu - \gamma \sqrt{s})^{1/2 + \a}} (\frac{1}{\sqrt{\nu}} \mathbbm{1}_{\{0 \leq t \leq {\nu^2}/\gamma^2\}} + 1) ds) \cdot 
 \sup_{s \in [0,t]}\normmm{\omega(s)}^2_s\\
=:& \tilde{M} + (J_1 + J_2 + J_3 + J_4)  \cdot 
 \sup_{s \in [0,t]}\normmm{\omega(s)}^2_s.
\end{aligned}
\end{equation*}
Recalling $\beta < 1/4$, there is
\begin{equation*}
J_1 \lesssim \frac{1}{\gamma^{1-2\a}},\ 
J_2 \lesssim \frac{1}{\gamma^\a},\ 
J_3 \lesssim \frac{1}{\gamma},\ 
J_4 \lesssim \frac{1}{\gamma},
\end{equation*}
whose proof may be found in the Appendix. Therefore, we obtain
\begin{equation}\label{eqX2}
\sum_{i+j\leq 1} \Vert \partial_x^i (y\partial_y)^j \omega(t)\Vert_{X_{\mu,t}} \lesssim \tilde{M} +  \frac{1}{\sqrt{\gamma}}  \cdot  \sup_{s \in [0,t]}\normmm{\omega(s)}^2_s.
\end{equation}
Combining \eqref{eqX1} and \eqref{eqX2}, we have
\begin{equation}\label{eqX}
\Vert \omega(t)\Vert_{X(t)} \lesssim \tilde{M} +  \frac{1}{{\gamma^{\min\{ \a, 1-2\a \}}}}  \sup_{s \in [0,t]}\normmm{\omega(s)}^2_s.
\end{equation}
From \cite{KukVicWan20}, there is
\begin{equation}\label{eqY}
\Vert  \omega(t)\Vert_{Y(t)} \lesssim \bar{M}+ \frac{1}{\sqrt{\gamma}} \sup_{s \in [0,t]}\normmm{\omega(s)}^2_s.
\end{equation}
Let
\begin{equation*}
M'=\sum_{i+j \leq 3} \Vert \partial_x \partial_y \omega_0 \Vert_{L^2(y \geq 1/4)}.
\end{equation*}
In view of the definition of $X$ and $Y$ norms, we have
\begin{equation*}
M' \lesssim M.
\end{equation*}
Applying Lemma \ref{Sobolev} yields 
\begin{equation}\label{eqSobolev}
\Vert \omega(t) \Vert_Z \lesssim \left(M' + \frac{1+ \sup_{t \in [0,\frac{\mu_0}{2 \gamma}]} \normmm{\omega(t)}_t^{3/2}}{\sqrt{\gamma}}\right) e^{\frac{C}{\gamma}(1+ \sup_{t \in [0,\frac{\mu_0}{2 \gamma}]} \normmm{\omega(t)}_t)}.
\end{equation}
Combining  \eqref{eqX}, \eqref{eqY}, and \eqref{eqSobolev}, we have
\begin{equation*}
\begin{aligned}
\sup_{t \in [0,\frac{\mu_0}{2 \gamma}]} \normmm{\omega(t)}_t \leq& C(\tilde{M} + \bar{M}) + C\frac{1}{\gamma^{\min\{ \a, 1-2\a \}}}\sup_{t \in [0,\frac{\mu_0}{2 \gamma}]} \normmm{\omega(t)}_t\\ 
&+ C\left(M' + \frac{1+ \sup_{t \in [0,\frac{\mu_0}{2 \gamma}]} \normmm{\omega(t)}_t^{3/2}}{\sqrt{\gamma}}\right) e^{\frac{C}{\gamma}(1+ \sup_{t \in [0,\frac{\mu_0}{2 \gamma}]} \normmm{\omega(t)}_t)},
\end{aligned}
\end{equation*}
where $C\geq 1$ is a constant which depends only on $\mu_0$ and $\theta_0$. Using a standard barrier argument, one may show that if $\gamma$ is large enough, there is 
\begin{equation*}
\sup_{t \in [0,\frac{\mu_0}{2 \gamma}]} \normmm{\omega(t)}_t \leq2C(\tilde{M} + \bar{M} + M' +1),
\end{equation*}
which concludes the proof of the theorem.
\end{proof}

\subsection{Proof of Theorem \ref{thm1}}
The proof of Theorem \ref{thm1} is the same as that of Theorem 3.2 in \cite{KukVicWan20} with small modification. We include it here for the sake of completeness. 
\begin{proof}
Let $T>0$ be as in Theorem \ref{thm0}. In view of Kato's criterion, we only need to show that
\begin{equation}\label{Kato}
\lim_{\nu \to 0} \nu \int_0^T \int_{\mathbb{H}} |\nabla u|^2 dxdyds=0.
\end{equation}
Note that 
\begin{equation*}
\begin{aligned}
 \nu \int_0^T \int_{\mathbb{H}} |\nabla u|^2 dxdyds 
=& \nu  \int_0^T \int_{\mathbb{H}} |\omega|^2 dxdyds\\
=& \nu \int_0^T \int_{\{y \leq 1/2\}} |\omega|^2 dxdyds + \nu \int_0^T \int_{\{y \geq 1/2\}} |\omega|^2 dxdyds
\end{aligned}
\end{equation*}
Since $\sqrt{\nu} \lesssim w_t(y)$, there is
\begin{equation*}
\begin{aligned}
&\nu \int_0^T \int_{\{y \leq 1/2\}} |\omega|^2 dxdyds\\
\lesssim& \sqrt{\nu} \int_0^T \sum_{\xi} \Vert  e^{\e_0(1-y)_+|\xi|} w_s(y) \omega_{\xi}(s) \Vert_{L^{\infty}(y \leq 1/2)}  \Vert  e^{\e_0(1-y)_+|\xi|} \omega_{\xi}(s) \Vert_{L^{1}(y \leq 1/2)} ds\\
\lesssim& \sqrt{\nu} \int_0^T \Vert  \omega(s) \Vert_{X(s)} \Vert \omega(s) \Vert_{Y(s)} ds.
\end{aligned}
\end{equation*}
Moreover, we have
\begin{equation*}
 \nu \int_0^T \int_{\{y \geq 1/2\}} |\omega|^2 dxdyds \lesssim \nu \int_0^T \Vert \omega(s) \Vert_S^2 ds.
\end{equation*}
Therefore, combining the two estimate above, we have
\begin{equation*}
\nu \int_0^T \int_{\mathbb{H}} |\nabla u|^2 dxdyds \lesssim M \sqrt{\nu} + M^2 \nu,
\end{equation*}
and \eqref{Kato} follows.
\end{proof}

\appendix     
\section{Estimate of $I$ and $J$ terms}
In this section, we estimate $I_1$ to $I_4$ in the proof of Theorem \ref{thm0}. All the notations are the same as those in the proof.
For the $I_1$ term, we  let $s= t \sin^2\theta$, $a= (\mu_0-\mu)/\gamma \sqrt{t}$ to get
\begin{equation*}
\begin{aligned}
\abs{I_1}=&\abs{\int_0^t \frac{1}{(\mu_0- \mu -\gamma \sqrt{s})^{1+\a}} \frac{1}{\sqrt{t-s}} \frac{1}{\sqrt{\nu}}ds \cdot  \mathbbm{1}_{\{0 \leq t \leq {\nu^2}/\gamma^2\}}} \\
\lesssim& \int_0^{\pi/2} \frac{1}{(\mu_0- \mu - \gamma \sqrt{t} \sin \theta)^{1+ \a}} \frac{1}{\sqrt{t} \cos \theta} \frac{1}{\sqrt{\nu}} t \sin \theta  \cos \theta d \theta \cdot  \mathbbm{1}_{\{0 \leq t \leq {\nu^2}/\gamma^2\}} \\
\lesssim&  \int_0^{\pi/2} \frac{1}{(\frac{\mu_0- \mu}{\gamma \sqrt{t}} - \sin \theta)^{1+ \a}}  \frac{1}{\gamma^{1+\a} \sqrt{t}^{\a}}\frac{1}{\sqrt{\nu}} \sin \theta d \theta \cdot  \mathbbm{1}_{\{0 \leq t \leq {\nu^2}/\gamma^2\}} \\
\lesssim&  \int_0^{\pi/2} \frac{1}{(a - \sin \theta)} \frac{(\gamma \sqrt{t})^{\a}}{(\mu_0-\mu -\gamma \sqrt{t})^\a} \frac{1}{\gamma^{1+\a} \sqrt{t}^{\a}}\frac{1}{\sqrt{\nu}} \sin \theta d \theta \cdot  \mathbbm{1}_{\{0 \leq t \leq {\nu^2}/\gamma^2\}} \\
\lesssim&  \int_0^{\pi/2} \frac{1}{(a - \sin \theta)} \sin \theta d \theta  \frac{1}{(\mu_0-\mu -\gamma \sqrt{t})^\a} \frac{1}{\gamma} \frac{1}{\sqrt{\nu}}\cdot  \mathbbm{1}_{\{0 \leq t \leq {\nu^2}/\gamma^2\}}. 
\end{aligned}
\end{equation*}
By a change of variable, we further arrive at
\begin{equation*}
\begin{aligned}
 \int_0^{\pi/2} \frac{1}{(a - \sin \theta)} \sin \theta d \theta
\lesssim \int_0^1 \frac{1}{a-x} \frac{x}{\sqrt{1-x^2}} dx
\lesssim  \int_0^1 \frac{1}{a-x} \frac{1}{\sqrt{1-x}} dx.
\end{aligned}
\end{equation*}
Setting $\sqrt{1-x}=v$, we obtain
\begin{equation*}
\begin{aligned}
\int_0^{\pi/2} \frac{1}{(a - \sin \theta)} \sin \theta d \theta \lesssim& \int_0^1 \frac{1}{a-1+v^2} dv \lesssim& \frac{1}{\sqrt{a-1}} \arctan(\frac{1}{\sqrt{a-1}}).
\end{aligned}
\end{equation*}
Therefore, $I_1$ may be bounded as
\begin{equation*}
\begin{aligned}
I_1 
\lesssim& \frac{1}{\sqrt{a-1}}  \frac{1}{(\mu_0-\mu -\gamma \sqrt{t})^\a} \frac{1}{\gamma} \frac{1}{\sqrt{\nu}}\cdot  \mathbbm{1}_{\{0 \leq t \leq {\nu^2}/\gamma^2\}} \\
\lesssim&  \frac{\sqrt{\gamma \sqrt{t}}}{(\mu_0-\mu -\gamma \sqrt{t})^{1/2+\a}} \frac{1}{\gamma} \frac{1}{\sqrt{\nu}}\cdot  \mathbbm{1}_{\{0 \leq t \leq {\nu^2}/\gamma^2\}}\\ 
\lesssim&  \frac{1}{(\mu_0-\mu -\gamma \sqrt{t})^{1/2+\a}} \frac{1}{\sqrt{\gamma}}.
\end{aligned}
\end{equation*}
 Similarly for $I_2$, we have
\begin{equation*}
\begin{aligned}
&\int_0^{t} \frac{1}{(\mu_0- \mu -\gamma \sqrt{s})^{1+ \a}} \frac{1}{\sqrt{t-s}} \frac{1}{s^\beta} ds\\
\lesssim&  \int_0^{\pi/2} \frac{1}{(a - \sin \theta)} \sin \theta \frac{1}{(t \sin^2 \theta)^\beta} d \theta  \frac{1}{(\mu_0-\mu -\gamma \sqrt{t})^\a} \frac{1}{\gamma} \\
\lesssim&  \int_0^{\pi/2} \frac{1}{(a - \sin \theta)} \sin^{1-2\beta}(\theta) d \theta \frac{1}{t^\beta} \frac{1}{(\mu_0-\mu -\gamma \sqrt{t})^\a} \frac{1}{\gamma}. \\
\end{aligned}
\end{equation*}
A change of variable gives
\begin{equation*}
\begin{aligned}
&\int_0^{t} \frac{1}{(\mu_0- \mu -\gamma \sqrt{s})^{1+ \a}} \frac{1}{\sqrt{t-s}} \frac{1}{s^\beta} ds\\
\lesssim& \int_0^{1} \frac{1}{(a - x)} x^{-2\beta}\frac{x}{\sqrt{1-x^2}} dx \frac{1}{t^\beta} \frac{1}{(\mu_0-\mu -\gamma \sqrt{t})^\a} \frac{1}{\gamma} \\
\lesssim& \int_0^{1} \frac{1}{(a - x)} x^{-2\beta}\frac{1}{\sqrt{1-x}} dx \frac{1}{t^\beta} \frac{1}{(\mu_0-\mu -\gamma \sqrt{t})^\a} \frac{1}{\gamma} \\
\lesssim& \int_0^1 \frac{1}{a-1+v^2} \frac{1}{(1-v^2)^{2\beta}} dv\frac{1}{t^\beta} \frac{1}{(\mu_0-\mu -\gamma \sqrt{t})^\a} \frac{1}{\gamma} \\
\lesssim& \int_0^1 \frac{1}{a-1+v^2} \frac{1}{(1-v)^{2\beta}} dv\frac{1}{t^\beta} \frac{1}{(\mu_0-\mu -\gamma \sqrt{t})^\a} \frac{1}{\gamma},
\end{aligned}
\end{equation*}
and the integral above may be further estimated by
\begin{equation*}
\begin{aligned}
& \int_0^1 \frac{1}{a-1+v^2} \frac{1}{(1-v)^{2\beta}} dv\\
\lesssim& \frac{1}{a-1} \int_0^1 \frac{1}{1+(\frac{v}{\sqrt{a-1}})^2} \frac{1}{(1-v)^{2\beta}} d\frac{v}{\sqrt{a-1}}\\
\lesssim& \frac{\sqrt{\gamma \sqrt{t}}}{(\mu_0- \mu -\gamma \sqrt{t})^{1/2}} \int_0^{\frac{1}{\sqrt{a-1}}} \frac{1}{1+ y^2} \frac{1}{(1-\sqrt{a-1} y)^{2 \beta}} dy\\
\lesssim& \frac{\sqrt{\gamma \sqrt{t}}}{(\mu_0- \mu -\gamma \sqrt{t})^{1/2}} (\int_0^{\frac{1}{2 \sqrt{a-1}}} \frac{1}{1+ y^2} \frac{1}{(1-\sqrt{a-1} y)^{2 \beta}} dy \\
&+ \int_{\frac{1}{2 \sqrt{a-1}}}^{\frac{1}{ \sqrt{a-1}}} \frac{1}{1+ y^2} \frac{1}{(1-\sqrt{a-1} y)^{2 \beta}} dy)\\
\lesssim&\frac{\sqrt{\gamma \sqrt{t}}}{(\mu_0- \mu -\gamma \sqrt{t})^{1/2}} (\int_0^{\frac{1}{2 \sqrt{a-1}}} \frac{1}{1+ y^2} dy  + \frac{4a-4}{4a-3} \int_{\frac{1}{2 \sqrt{a-1}}}^{\frac{1}{ \sqrt{a-1}}}  \frac{1}{(1-\sqrt{a-1} y)^{2 \beta}} dy)\\
\lesssim& \frac{\sqrt{\gamma \sqrt{t}}}{(\mu_0- \mu -\gamma \sqrt{t})^{1/2}} (\frac{\pi}{2} + \frac{4\sqrt{a-1}}{4a-3})\lesssim \frac{\sqrt{\gamma \sqrt{t}}}{(\mu_0- \mu -\gamma \sqrt{t})^{1/2}}.
\end{aligned}
\end{equation*}
Therefore, we arrive at
\begin{equation*}
\begin{aligned}
\abs{I_2} \lesssim&  \frac{\sqrt{\gamma \sqrt{t}}}{(\mu_0- \mu -\gamma \sqrt{t})^{1/2}}\frac{1}{t^\beta} \frac{1}{(\mu_0-\mu -\gamma \sqrt{t})^\a} \frac{1}{\gamma} \lesssim \frac{1}{\sqrt{\gamma}} \frac{1}{(\mu_0-\mu -\gamma \sqrt{t})} t^{1/4-\beta}.
\end{aligned}
\end{equation*}
For  $I_3$ and $I_4$, letting $\sqrt{s}=u$, we have
\begin{equation*}
\begin{aligned}
\abs{I_3} =&\abs{\frac{1}{\sqrt{\nu}} \mathbbm{1}_{\{0 \leq t \leq {\nu^2}/\gamma^2\}}\int_0^{\sqrt{t}} \frac{1}{(\mu_0 - \mu - \gamma u)^{3/2 + \a}} udu} \\
\lesssim& \frac{\sqrt{t}}{\sqrt{\nu}} \frac{1}{(\mu_0 - \mu - \gamma \sqrt{t})^{1/2 + \a}} \mathbbm{1}_{\{0 \leq t \leq {\nu^2}/\gamma^2\}}\lesssim  \frac{\nu}{\sqrt{\nu} \gamma^2} \frac{1}{(\mu_0 - \mu - \gamma \sqrt{t})^{1/2 + \a}}\\
\lesssim&  \frac{1}{\sqrt{\gamma}} \frac{1}{(\mu_0 - \mu - \gamma \sqrt{t})^{1/2 + \a}},\\
\end{aligned}
\end{equation*}
and
\begin{equation*}
\begin{aligned}
\abs{I_4}=&\abs{\int_0^{\sqrt{t}} \frac{1}{(\mu_0 - \mu - \gamma u)^{3/2 + \a}} udu} \lesssim  \frac{1}{\sqrt{\gamma}} \frac{1}{(\mu_0 - \mu - \gamma \sqrt{t})^{1/2 + \a}}.\\
\end{aligned}
\end{equation*}
The $J_1$ to $J_4$ terms in the proof of Theorem \ref{thm0} are estimated in a similar manner and we omit further details.

%
%

\bibliographystyle{alpha}
\bibliography{FeiBib}

\end{document}